\documentclass[12pt]{amsart}
\usepackage[utf8]{inputenc}

\usepackage{comment}
\usepackage[english]{babel}
\usepackage[T1]{fontenc}
\usepackage{todonotes}
\usepackage{bm}
\usepackage{upgreek}

\usepackage{enumitem}
\usepackage{amssymb,amsfonts,stmaryrd,mathabx}
\usepackage{amsmath,latexsym,amsthm,mathtools,bbm}
\usepackage{pgf,tikz} 
\usepackage{subcaption}
\usepackage[colorlinks,cite color=green,link color=purple]{hyperref}
\usepackage{todonotes}
\usepackage{cases}
\usepackage[thinlines]{easytable}
\usepackage{bm}
\usepackage[export]{adjustbox}
\usepackage{appendix}

\usepackage{thmtools}
\usepackage[capitalize,noabbrev,nameinlink]{cleveref}
\newtheorem{thm}{Theorem}[section]
\newtheorem{cor}[thm]{Corollary}
\newtheorem{lem}[thm]{Lemma}
\newtheorem{prop}[thm]{Proposition}

\theoremstyle{definition}
\newtheorem{definition}[thm]{Definition}
\newtheorem{remark}[thm]{Remark}
\newtheorem{example}[thm]{Example}
\newtheorem{notation}[thm]{Notation}

\AtBeginEnvironment{appendices}{\crefalias{section}{appendix}}



\newcommand{\bfx}{\boldsymbol{x}}

\DeclareMathOperator{\wt}{wt}

\usepackage{stackengine}
\usepackage{scalerel}

\DeclareMathOperator{\skipped}{skip}
\DeclareMathOperator{\free}{free}
\DeclareMathOperator{\emp}{empty}

\DeclareMathOperator{\inc}{inc}

\DeclareMathOperator{\SSYT}{SSYT}

\DeclareMathOperator{\pair}{pair}
\DeclareMathOperator{\rev}{rev}

\newcommand{\xx}{\mathbf x}

\newcommand{\mcP}{\mathcal P}

\newcommand{\ZZ}{\mathbb Z}
\newcommand{\QQ}{\mathbb Q}

\newcommand{\VV}{V_{\lambda}^*}

\newcommand{\mcG}{\mathcal{G}}
\newcommand{\Gbar}{\widebar{\mathcal{G}}}
\newcommand{\Qbar}{\widebar{Q}}
 
\newcommand{\NN}{\mathbb N}
\newcommand{\YY}{\mathbb Y}

\newcommand{\PP}{\mathbb P}

\newcommand{\PPone}[1]{\mathbb P^{(1)}_{#1}}
\newcommand{\PPtwo}[1]{\mathbb P^{(2)}_{#1}}
\newcommand{\mfp}{\mathfrak p}
\newcommand{\mfq}{\mathfrak q}

\newcommand{\tb}{c}
\allowdisplaybreaks

\usepackage[colorlinks,cite color=green,link color=purple]{hyperref}
\definecolor{green}{RGB}{43,92,47}
\definecolor{blue}{RGB}{40,68,104}
\definecolor{red}{RGB}{254, 113, 96}
\definecolor{purple}{RGB}{102,0,51}
\definecolor{gray}{RGB}{224,224,224}
\definecolor{lightpurple}{RGB}{255, 249, 242}
\definecolor{blue}{RGB}{40,68,104}
\hypersetup{colorlinks = true, urlcolor=blue}

\author[]{}

\title[A probabilistic interpretation for interpolation Macdonald polynomials]{A probabilistic interpretation for\\ interpolation Macdonald polynomials}
\date{\today}
\author{Houcine Ben Dali}
\address{Department of Mathematics,
Harvard University, Cambridge, MA, and Center for Mathematical Sciences and Applications, Harvard University, Cambridge, MA}
\email{bendali@math.harvard.edu}
\author{Lauren Kiyomi Williams}
\address{Department of Mathematics,
Harvard University, Cambridge, MA}
\email{williams@math.harvard.edu}

\begin{document}

\begin{abstract}
Previous work of Ayyer, Martin, and Williams gave a probabilistic interpretation of the Macdonald polynomials $P_{\lambda}(x_1,\dots,x_n;1,t)$ at $q=1$ in terms of a Markov chain called the \emph{multispecies $t$-Push TASEP}, a Markov chain involving particles of types $\lambda_1,\dots,\lambda_n$ hopping around a ring.  In particular, they showed that for each composition $\eta$ obtained by permuting the parts of $\lambda$, the stationary probability of being in state $\eta$ is proportional to the ASEP polynomial $F_{\eta}(x_1,\dots,x_n; 1,t)$, and the normalizing constant (or partition function) is $P_{\lambda}(x_1,\dots,x_n; 1,t)$.  There is an inhomogeneous generalization of Macdonald polynomials due to Knop and Sahi called \emph{interpolation Macdonald polynomials} $P^*_{\lambda}(x_1,\dots,x_n;q,t)$, as well as an inhomogeneous generalization of ASEP polynomials called \emph{interpolation ASEP polynomials} $F^*_{\eta}(x_1,\dots,x_n;q,t)$ that we introduced in previous work.  In this article 
we introduce a new Markov chain called the \emph{interpolation $t$-Push TASEP}, and show that its steady state probabilities and partition function are given by the interpolation ASEP polynomials and the interpolation Macdonald polynomial, evaluated at $q=1$.  This generalizes the previous result of Ayyer, Martin, and Williams.

\end{abstract}

\maketitle

\begin{center}
\begin{minipage}{0.33\textwidth}
 \begin{center}
  \emph{We could discover\\ the polynomials of \\ 
  Knop and Sahi, by \\ studying distributions \\
  of particles on a ring}
     \end{center}
\end{minipage}
\end{center}

 \setcounter{tocdepth}{1}
 \tableofcontents

\section{Introduction}

Several recent works have given an interpretation of Macdonald polynomials in terms of stationary distributions of interacting particle models.
For example, in \cite{CantinideGierWheeler2015}, Cantini, de Gier, and Wheeler showed that the partition function of the \emph{multispecies asymmetric simple exclusion process (ASEP) on a ring} involving particles of types $\lambda_1,\dots,\lambda_n$ is equal to a Macdonald polynomial 
$P_{\lambda}(x_1,\dots,x_n;q,t)$ specialized at 
$x_1=\dots=x_n=q=1$; moreover, for each composition $\eta$ obtained by permuting the parts of $\lambda$, the  steady state probability of state $\eta$ is given by  the \emph{ASEP polynomial} $F_{\eta}$ evaluated at 
$x_1=\dots=x_n=q=1$.  (An ASEP polynomial is a certain \emph{permuted basement Macdonald polynomial}, see \cite[Proposition 4.1]{CorteelMandelshtamWilliams2022}.)

After this work, the natural question was whether one can give a parallel result which incorporates the parameters $x_1,\dots,x_n$ into the Markov chain.  This question was addressed by Ayyer, Martin, and Williams 
in   \cite{AyyerMartinWilliams2025}, who showed that the steady state probabilities and partition function of the  \emph{multispecies $t$-Push totally asymmetric simple exclusion process (TASEP)} -- a Markov chain in which there are site-dependent jump rates involving the $x_i$'s -- are given by the ASEP polynomials $F_{\eta}$ and Macdonald polynomial $P_{\lambda}$ evaluated at $q=1$.\footnote{See \cite{Defant2024} for an alternative probabilistic interpretation of these polynomials at $q=1$.}

There is an inhomogeneous generalization of Macdonald polynomials due to Knop and Sahi \cite{Knop1997b, Sahi1996} called \emph{interpolation Macdonald polynomials} $P^*_{\lambda}(x_1,\dots,x_n;q,t)$, as well as an inhomogeneous generalization of ASEP polynomials called \emph{interpolation ASEP polynomials} $F^*_{\eta}(x_1,\dots,x_n;q,t)$ that we introduced in our previous work \cite{BenDaliWilliams2025}; we also gave a combinatorial formula for the interpolation ASEP and Macdonald polynomials in terms of \emph{signed multiline queues}.  
In this article 
we introduce a new Markov chain called the \emph{interpolation $t$-Push TASEP}
(see \cref{def:intpushTASEP}), which can be seen as a generalization of the $t$-Push TASEP (see \cref{rem:generalization}).  Our main theorem (see \cref{thm:main}) says that its steady state probabilities and partition function are given by the interpolation ASEP polynomials and interpolation Macdonald polynomial evaluated at $q=1$.  This generalizes the previous result of Ayyer, Martin, and Williams.

We note that this paper addresses a (generalization of a) question which
was posed by L.W. in the recent preprint
``First Proof'' \cite{FirstProof}.  

\subsection{Interpolation polynomials}
Interpolation Macdonald polynomials and interpolation ASEP polynomials are defined by vanishing conditions, as we now explain.

Given a composition $\mu=(\mu_1,\dots,\mu_n)\in\NN^n$, we define
\begin{align} \label{eq:k}
k_i(\mu)&:=\#\{j:j<i \text{ and }\mu_j>\mu_i\}+\#\{j:j>i \text{ and }\mu_j\geq \mu_i\}, \text{ and }\\\label{eq:k2}
\overline\mu&:=\left(q^{\mu_1} t^{-k_1(\mu)},\dots,q^{\mu_n} t^{-k_n(\mu)}\right).
\end{align}

\emph{Interpolation Macdonald polynomials}, defined by Knop and Sahi \cite{Knop1997b, Sahi1996}, can be viewed as an inhomogeneous generalization of Macdonald polynomials.
\begin{thm}\cite{Knop1997b, Sahi1996}\label{def:intMacdonald}
For each partition $\lambda=(\lambda_1,\dots,\lambda_n)$, there is a unique inhomogeneous symmetric polynomial $P^*_{\lambda} =P_{\lambda}^*(\xx;q,t) = P_{\lambda}^*(x_1,\dots,x_n;q,t)$ 
called the \emph{interpolation Macdonald polynomial} such that 
\begin{itemize}
    \item the coefficient $[m_{\lambda}]P_{\lambda}^*$ of the 
    monomial symmetric polynomial $m_{\lambda}$ in $P_{\lambda}^*$ is $1$,
    \item $P_{\lambda}^*(\overline{\nu}) = 0$ for each partition $\nu \neq \lambda$ with $|\nu| \leq |\lambda|$.
\end{itemize}
Moreover, the top homogeneous component of 
$P_{\lambda}^*$ is the usual Macdonald polynomial
$P_{\lambda}.$
\end{thm}

Just as interpolation Macdonald polynomials are an inhomogeneous generalization of Macdonald polynomials, 
\emph{interpolation ASEP polynomials}, defined in \cite{BenDaliWilliams2025}, can be viewed as an inhomogeneous generalization of ASEP polynomials. 

Define $\mcP_n^{(d)}$ to be the space of polynomials in $n$ variables with degree at most $d$.

\begin{thm}[{\cite[Theorem 2.17]{BenDaliWilliams2025}}]
    Let $\lambda=(\lambda_1,\dots,\lambda_n)$ be a partition of size
    $|\lambda|=\lambda_1+ \dots + \lambda_n=d$.  For each composition $\mu\in S_n(\lambda)$, there is a unique polynomial $F^*_\mu(x_1,\dots,x_n)\in \mcP_n^{(d)}$, called the \emph{interpolation ASEP polynomial}, such that:
    \begin{itemize}
     \item for $\tau \in S_n(\lambda)$, we have 
     \begin{equation}\label{eq:normalization_F}
         [x^\tau]F^*_\mu =\delta_{\tau,\mu}.
     \end{equation}
        \item for any composition $\nu$ such that $|\nu|\leq |\mu|$ and $\nu\notin S_n(\lambda)$, we have
        $F^*_\mu(\widebar{\nu})=0$.
    \end{itemize}
Moreover, the top homogeneous component of $F^*_\mu$ is the ASEP polynomial $F_\mu$. 
\end{thm}

Just as the ASEP polynomials sum to a Macdonald polynomial, the interpolation ASEP polynomials sum to the interpolation Macdonald polynomial.  

\begin{prop}[{\cite[Proposition 2.15]{BenDaliWilliams2025}}]
\label{prop:symmetrization}For any partition $\lambda$, we have $$P^*_\lambda=\sum_{\mu\in S_n(\lambda)}F^*_\mu.$$
\end{prop}

\subsection{The Interpolation \texorpdfstring{$t$}{t}-Push TASEP}
We now define a Markov chain with state-space $S_n(\lambda)$ 
whose stationary distribution is given by the interpolation ASEP polynomials.  We call this Markov chain  the \emph{interpolation $t$-Push TASEP}, or the \emph{$t$-Push$^*$ TASEP.}

\begin{notation}\label{not:type}
Let $\lambda=(\lambda_1,\dots, \lambda_n)$
with $\lambda_1\geq\dots\geq\lambda_n\geq 0$ be a partition. We can
describe such a partition by its vector of types
$\mathbf{m}=(m_0, m_1, \dots, m_L)$, where
$m_i=m_i(\lambda):=\#\{j:\lambda_j=i\}$,
and $L$ is the largest part that occurs.
We thus sometimes denote our partition by
$\lambda=(L^{m_L}, \dots, 1^{m_1}, 0^{m_0})$.
Note that $\sum_{i=0}^L m_i=n$.  We also write
$M_i:=m_i+ m_{i+1} + \dots + m_L$.
\end{notation}

 For $1\leq k\leq n$, we define the following elements in $\QQ(t,x_1,\dots,x_n)$.
$$\mfp_k:=\frac{t^{-n+1}(1-t)}{x_k-t^{-n+2}}, \qquad \text{and}\qquad \mfq_k:= \frac{(1-t)x_k}{x_k-t^{-n+2}}.$$
We then have
\begin{equation}\label{eq:1minuspk}
1-\mfp_k=\frac{x_k-t^{-n+1}}{x_k-t^{-n+2}} \qquad \text{and}\qquad 1-\mfq_k:= \frac{t x_k-t^{-n+2}}{x_k-t^{-n+2}}.
\end{equation}

If  $0<t<1$ and $x_i>t^{-n+1}$ for $1\leq i\leq n$, then $\mfp_k$ and $\mfq_k$ are probabilities. 

\begin{definition}[The interpolation $t$-Push TASEP]\label{def:intpushTASEP}
Fix a partition $\lambda=(\lambda_1,\dots,\lambda_n)$ with at least one part of size~0.
The \emph{interpolation $t$-Push TASEP} with 
\emph{(particle) content $\lambda$} is a Markov chain
on $S_n(\lambda)$; we think of its states as configurations of particles on a ring labeled by $\lambda_1,\dots, \lambda_n$, where state $\eta$ corresponds to having  a particle labeled $\eta_j$ at position $j$.  Moreover, there is a \emph{bell} attached to each particle.
The transitions from $\eta\in S_n(\lambda)$ are obtained as follows.
\begin{enumerate}[label=(Step \arabic*)]
   \item[(Step 0)]\label{Step0} The bell at position $j$ rings with probability 
   $$P_j
   =\frac{\prod_{k<j}\left(x_k-\frac{1}{t^{n-2}}\right)\prod_{k>j}\left(x_k-\frac{1}{t^{n-1}}\right)}{e^*_{n-1}(\bfx;t)}, $$
   where 
   $e^*_{n-1}(\bfx;t)=\sum_{j=1}^n\prod_{k<j}\left(x_k-\frac{1}{t^{n-2}}\right)\prod_{k>j}\left(x_k-\frac{1}{t^{n-1}}\right)$.
      \item\label{Step1}(\emph{$t$-Push TASEP})
        The particle at position $j$, say with label $a$, is activated, and starts traveling clockwise according to the rules of the \emph{$t$-Push TASEP}.  That is, suppose  there are $m$ ``weaker'' particles in the system, i.e. particles whose species is less than $a$ (including vacancies). 
   Then with probability $\frac{t^{k-1}}{[m]_t}$ the activated particle will move to the location of the $k$th of these weaker particles.  If this location contains a particle, then that particle becomes active, and chooses a weaker particle to displace in the same way.  The procedure continues until the active particle arrives at a vacancy.  All these transitions occur simultaneously.   
   At the end of this step, position $j$ is vacant, and 
   we regard this vacancy as a particle labeled $a:=0$.

  An equivalent way to describe the above process is that 
   each time the active particle passes a site with 
   a weaker particle, it continues to move with probability $t$, and settles at that site with probability $(1-t)$, displacing and activating the particle 
   that is located there.
   If it passes the $m$th such site, then it continues cyclically around the ring.  Note that an active particle will choose the $k$th available option with probability
   $$(1-t)(t^{k-1}+t^{k-1+m} + t^{k-1+2m}+ \dots) = 
   \frac{t^{k-1}}{1+t+\dots + t^{m-1}}.$$
   
   \item\label{Step2}(\emph{Return to the bell} TASEP)
   The particle labeled $a:=0$ 
   now goes to position $1$ and starts traveling clockwise.   When it gets to site $k$ for $1\leq k\leq j-1$ containing a particle with label $b \geq 0$, it 
   skips over that site with probability
   \begin{equation}\label{eq:probskip}
        \begin{cases}
       1-\mfp_k &\text{ if $b\geq a$},\\
       1-\mfq_k &\text{ if $b<a$}
   \end{cases}
   \end{equation}
   and 
   settles at that site (displacing and activating the site's particle) 
   with probability
    \begin{equation}\label{eq:prob}
        \begin{cases}
       \mfp_k &\text{ if $b\geq a$},\\
       \mfq_k &\text{ if $b<a$}.
   \end{cases}
   \end{equation}
Once it activates a new particle, the old particle settles at site $k$ and the new active particle continues to travel clockwise towards position $j$, activating a new particle according to \eqref{eq:prob}.   
   The active particle stops once it displaces/activates another particle or arrives at position $j$, in which case it settles in position $j$. 
\end{enumerate}
The resulting configuration of particles is denoted
$\nu$, and we denote the transition probability by $\PP_\lambda(\eta,\nu)=\PP(\eta,\nu)$.

Moreover, we let $\PPone{\lambda,j}=\PPone{j}$ and $\PPtwo{\lambda,j}=\PPtwo{j}$ denote the transition probabilities associated with \ref{Step1} and \ref{Step2}, respectively. We then have, for $\mu,\nu\in S_n(\lambda)$, 
$$\PP(\mu,\nu)=\sum_{1\leq j\leq n}P_j\sum_{\rho\in S_n(\lambda):\rho_j=0}\PPone{j}(\mu,\rho)\PPtwo{j}(\rho,\nu).$$
\end{definition}
One can check that this Markov chain is irreducible; this implies that it has a unique stationary distribution $(\pi^*_\eta)_{\eta\in S_n(\lambda)}$.

\begin{remark}
   One may notice that the main differences between a classical step in $t$-Push TASEP (Step 1), and the new dynamics we introduce here in (Step 2) are the following: the activated particle can now displace a particle with a lower \emph{or} higher label. Moreover, the activated particle will stop with probability $1$ in position $j$; in particular, particles can no longer wrap several times around the ring.
    \end{remark}

\begin{figure}[t]
    \centering
    \includegraphics[height=0.25\linewidth]{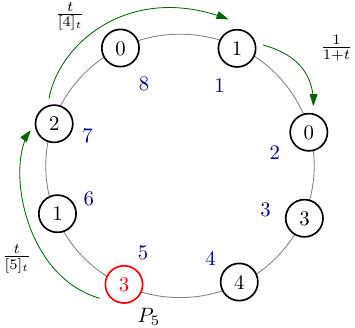}
    \includegraphics[height=0.26\linewidth]{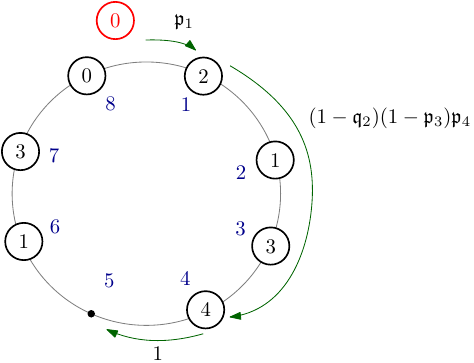}
    \includegraphics[height=0.23\linewidth]{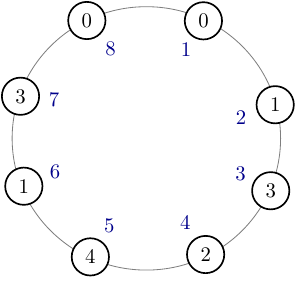}
    \caption{An example of the interpolation $t$-Push TASEP dynamics with $n=8$ and $\eta=(1,0,3,4,3,1,2,0)$. The figure on the left illustrates \ref{Step0} and \ref{Step1}: the bell rings at position $j=5$, then balls in positions 5, 7 and 1 and 2 are activated successively (the quantities next to the arrows give the transition rates). The figure in the middle illustrates \ref{Step2}: the ball at position $j=5$ is activated and travels clockwise starting from position 1. Then balls in positions 1, 2, 4<$j$ are activated successively. The figure on the right gives the final configuration.
    }
    \label{fig:Markov_chain}
\end{figure}
\begin{remark}\label{rmq:twin_particles}
Note that our assumption in \cref{def:intpushTASEP} that $\lambda$ has at least one part of size $0$ is not a restriction.  If $\lambda$ does not satisfy this property, then we can replace $\lambda$ by the partition $\tilde{\lambda} = (\lambda_1-\lambda_n, \lambda_2-\lambda_n,\dots, \lambda_{n-1}-\lambda_n, 0)$ and use the associated Markov chain.

Note also that the interpolation $t$-Push TASEP does not depend on whether a particle labeled $a$ can displace another particle labeled $a$, since they are ``identical''.
\end{remark}

\begin{example}
    We give in \cref{fig:transition_graph} the transition graph of Steps 1 and 2 in the $t$-Push$^*$ TASEP for $\lambda=(2,1,0)$, assuming that the bell rings at $j=3$ in Step 0. One may notice that, by definition, the transition probabilities in Step 2 are only defined from states $\mu$ for which $\mu_3=0$. For clarity, the loops (corresponding to transitions which do not change the configuration) are not represented in the graph. 
\end{example}

\begin{figure}[t]
    \centering
    \includegraphics[width=0.5\linewidth]{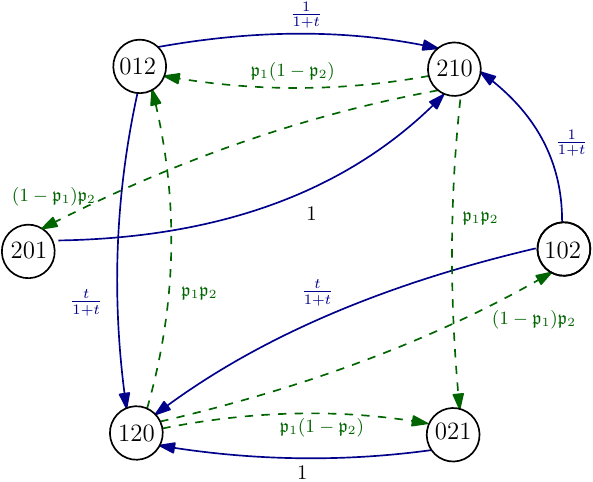}
    \caption{The transition graph of \ref{Step1} and \ref{Step2} of the $t$-Push$^*$ TASEP for $\lambda=(2,1,0)$ when $j=3$. The transition edges corresponding to \ref{Step1} (respectively \ref{Step2}) are represented in plain edges (respectively dashed edges).} 
    \label{fig:transition_graph}
\end{figure}

\begin{remark}\label{rem:generalization}
One can obtain the $t$-Push TASEP model by taking $x_i\gg t$. More formally, the asymptotics when $x_i\rightarrow \infty$ of the probabilities above give
$$P_j\sim \frac{1}{x_j}\left(\sum_{1\leq k\leq n}\frac{1}{x_k}\right)^{-1},\qquad   \mfp_k\rightarrow 0,\quad \text{and }\quad\mfq_k\rightarrow 1-t $$

Note in particular that the second asymptotic implies that nothing changes in \ref{Step2} (there is no particle activated, and the process stops when we arrive at position $j$).
\end{remark}

Our main theorem is the following.
\begin{thm}\label{thm:main}
    In the interpolation $t$-Push TASEP with content
    $\lambda=(\lambda_1,\dots,\lambda_n)$ and parameters
    $\bfx=(x_1,\dots,x_n)$ and $t$, the stationary probability of a configuration $\mu\in S_n(\lambda)$ is given by 
    $$\pi^*_{\lambda}(\mu) = \frac{F^*_{\mu}(\bfx; 1, t)}{P^*_{\lambda}(\bfx; 1, t)},$$
    where $F^*_{\mu}(\bfx;q,t)$ is the interpolation ASEP polynomial, and $P^*_{\lambda}(\bfx; q, t)$ is the interpolation Macdonald polynomial.
\end{thm}

A combinatorial formula for the interpolation ASEP polynomials in terms of \emph{signed multiline queues} was given in \cite[Theorem 1.15]{BenDaliWilliams2025}.  Combining this with \cref{thm:main}, we obtain the following corollary.

\begin{cor}
Fix a partition $\lambda=(\lambda_1,\dots,\lambda_n)$, and consider the signed multiline queues of type $\mu$, for $\mu\in S_n(\lambda),$ with parameters $x_1,\dots,x_n,t$ and $q=1$.
The distribution of the bottom line of the signed multiline queues is the same as the stationary distribution of the interpolation $t$-Push TASEP with content $\lambda$.
\end{cor}

\subsection{Strategy of Proof}
Let us briefly describe the strategy for proving \cref{thm:main}, which is similar to the strategy used in 
\cite{AyyerMartinWilliams2025} for proving the analogous result connecting the $t$-Push TASEP to ASEP polynomials.
We start by giving a combinatorial proof of the theorem
when $\lambda$ has distinct parts.  We use the fact that 
interpolation ASEP polynomials are generating functions for 
signed multiline queues \cite{BenDaliWilliams2025}, 
and that signed multiline queues can be built by stacking alternating layers of ``classic'' and ``signed'' layers of pairings of balls.  We also use the fact that
$\PPone{\lambda,j}(\eta, \kappa)$ gives the probability of obtaining
$\kappa$ as Row $1'$ of a signed multiline queue with content $\lambda$, given that Row $2$ is $\eta$ and the vacancy in Row $1'$ is at site $j$.  Similarly, we show that 
$\PPtwo{\lambda,j}(\kappa, \nu)$ gives the probability of obtaining $\nu$ as Row $1$ of a signed multiline queue with content $\lambda$, given that Row $1'$ is $\kappa$ and the vacancy in Row $1'$ is at site $j$.  This allows us to directly connect the transition probabilities of our Markov chain to the generating functions for signed multiline queues, which give the interpolation ASEP polynomials.

Having proved \cref{thm:main} for partitions with distinct parts, we generalize the result to all partitions, by 
using the fact that ``recoloring'' our particles in the Markov chain via a weakly order-preserving function $\phi:\NN \to \NN$ gives rise to a projection or lumping from a ``finer'' to a ``coarser'' version of the interpolation $t$-Push TASEP.  This allows us to compute the stationary distribution of a ``coarser'' interpolation $t$-Push TASEP associated to a partition with repeated parts, using the stationary distribution coming from a partition with distinct parts.  Meanwhile, we prove that interpolation ASEP polynomials at $q=1$ satisfy a corresponding identity that gives a compatibility relation when one applies a weakly order-preserving function.

\bigskip
\subsection{On the multispecies ASEP}
As mentioned earlier,   Cantini, de Gier, and Wheeler \cite{CantinideGierWheeler2015} showed that
the specialization $P_{\lambda}(x_1=\dots=x_n=1;q=1,t)$ of the homogeneous Macdonald polynomial 
 is the partition function of the multispecies ASEP on a ring. This result follows readily from the fact that the homogeneous ASEP polynomials $F_\mu$ are characterized by a family of equations called the \emph{qKZ equations} (see e.g. \cite[Equations (17) to (19)]{CantinideGierWheeler2015}).  In the interpolation case, however, 
 we cannot hope for a straightforward analogue of the multispecies ASEP result, because 
 the interpolation polynomials are not in general  positive at $x_1=\dots=x_n=q=1$.

Another obstacle to finding an analogous probabilistic interpretation is the fact that interpolation ASEP polynomials $F_\mu^*$ do not satisfy all the qKZ equations (more precisely, they lack the circular symmetry property). It would then be interesting to find an algebraic characterization of the interpolation polynomials $F_\mu^*$ from which the qKZ equations can be recovered by taking the top homogeneous part.

\medskip 

The structure of this paper is as follows. In \cref{sec:MLQ}, we review the notions of  two-line queues and signed two-line queues. \cref{sec:recoloring_Push_TASEP} is devoted to a recoloring property of the interpolation $t$-Push TASEP model. In \cref{sec:interpolation_polynomials}, we establish several properties of the interpolation ASEP polynomials at $q=1$.  We then prove the main theorem in \cref{sec:proof}. In \cref{sec:density}, we derive the density formula. Finally, in \cref{app:shape_permuting,app:permuted_basement,app:t_int_Schur}, we collect useful properties of interpolation ASEP polynomials, along with a Jacobi--Trudi identity for symmetric interpolation polynomials at $q=t$.

\bigskip
\noindent{\bf Acknowledgements:~}
HBD acknowledges support from the Center of Mathematical Sciences and Applications at Harvard University.
LW was supported by the National Science Foundation under Award No.
DMS-2152991.
Any opinions, findings, and conclusions or recommendations expressed in this material are
those of the author(s) and do not necessarily reflect the views of the National Science
Foundation.

\section{Multiline queues}\label{sec:MLQ}

\emph{Multiline queues} are a combinatorial object whose generating functions
encode Macdonald polynomials \cite{CorteelMandelshtamWilliams2022}; roughly speaking, multiline queues can be built by stacking \emph{two-line queues}.
\emph{Signed multiline queues}
generalize multiline queues, and are a combinatorial object whose generating functions encode interpolation ASEP polynomials \cite{BenDaliWilliams2025}; roughly speaking, signed multiline queues 
can be built by stacking alternating sequences of (classical) two-line queues and \emph{signed two-line queues}.  
In this section we define both 
(classical) two-line queues and signed two-line queues.  We will later explain how these two-line queues encode transitions in the 
interpolation $t$-Push TASEP.

\subsection{(Classical) two-line queues}

\begin{definition}\label{def:paired_system}
A \emph{paired ball system} is a $2 \times n$ array
in which each position is either occupied by a ball or is empty,
and there are at least as many balls on the bottom row as on the top;
moreover, each ball in the top row is paired with one in the bottom, and joined by a shortest strand traveling weakly left to right, allowing wrapping if necessary.
\end{definition}

\begin{definition}\label{def:two-row}
A \emph{two-line queue}\footnote{In \cite{CorteelMandelshtamWilliams2022,BenDaliWilliams2025}, these objects are referred to as generalized two-line queues. For simplicity, we will call them two-line queues throughout this paper.} is a paired ball system,   such that:
\begin{enumerate}
\item each ball is labeled by an element of $\ZZ^+$;
\item paired balls have the same label;
\item\label{GMLQ:item1} if a ball with label $a\in \ZZ^+$ in the top row has a ball directly underneath it labeled $a'$, then
        we must have $a'\geq a$, and if $a'=a$, the two balls must be paired to each other. 
\end{enumerate}
The bottom and top rows of a two-line queue are then represented by  compositions $\mu$ and $\nu$ in $\NN^n$, where an empty position is represented by $0$.  We refer to a pairing of two balls in the same column as a \emph{trivial pairing}.
\end{definition}

\begin{figure}[t]
    \centering
    \includegraphics[width=0.5\linewidth]{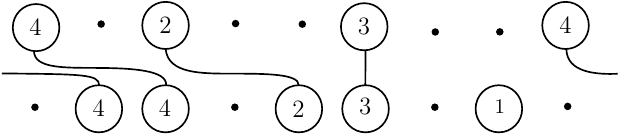}
    \caption{An example of a two-line queue in $\mathcal{Q}^{(4,0,2,0,0,3,0,0,4)}_{(0,4,4,0,2,3,0,1,0)}.$}
    \label{fig:two_line_queue}
\end{figure}

See \Cref{fig:two_line_queue} for an example of a two-line queue.

\begin{remark}\label{rem:MLQ}
    More generally, a \emph{multiline queue}
    is a vertical concatenation of 
    two-line queues, with rows labeled from $1$ to $L$ from bottom to top.
    We refer to a collection of linked pairs of balls as a \emph{strand},
    and we require that all balls in a strand whose top ball is in row $r$, are labeled $r$.
    See \cite{CorteelMandelshtamWilliams2022} for the precise definition.
\end{remark}

We associate to each pairing $p$ in a two-line queue $Q$ a weight $\wt(p)$ defined as follows: we read the balls in the top row in decreasing order of  their label; within a fixed  value, we read the balls from right to left.  Reading the balls in this order, we place the strands pairing the balls one by one.  The balls in the bottom row that have not yet been matched  are \emph{free}.  If a pairing $p$ matches a ball labeled $a$ in the top row and column $j$ to a ball labeled $a$ in the bottom row and column $k$, then the free balls in the bottom row and columns $j+1,j+2,\dots,k-1$ (where this set of columns is cyclically consecutive)
are  \emph{skipped}.  We then set 
\begin{equation}
\wt(p) = 
         (1-t) \frac{t^{\skipped(p)}}{1-t^{\free(p)}} 
\end{equation}
Having associated a weight to each nontrivial pairing,
we define $$\wt_{\pair}(Q) = \prod_p \wt(p),$$
where the product is over all nontrivial pairings of balls in $Q$.

We let  $\mathcal{Q}_\kappa^\eta$ denote the set of 
classical two-line queues with top row $\eta$ and bottom row $\kappa$, and we let
$a^\eta_{\kappa}$  
denote  the weight generating function of these two-line queues:\footnote{We note that in \cite{CorteelMandelshtamWilliams2022}, 
two-line queues were defined together with a weight function that depends on both $t$ and $q$; here we specialize to $q=1$.}
\begin{equation}\label{eq:generating_function0}
    a_\kappa^\eta=a_\kappa^\eta(t):=\sum_{Q\in\mathcal{Q}_\kappa^\eta}\wt_{\pair}(Q).
\end{equation}

\subsection{Signed two-line queues}

We review the notion of signed two-line queues introduced in \cite{BenDaliWilliams2025} to give a combinatorial formula for the interpolation ASEP polynomials $f^*_\mu(\bfx;q,t)$. We give here the case $q=1$ of this result.

\begin{definition}[\cite{BenDaliWilliams2025}]\label{def:signed_two-row}
A \emph{signed two-line queue} is a paired ball system (as in \cref{def:paired_system}), 
 such that:
\begin{enumerate}
 \item\label{GSMLQ:item1} Each pairing connects two balls with a shortest strand that travels either straight down or from left to right, and does not wrap around;
\item each ball in the top row is labeled by an element of $\ZZ\backslash\{0\}?$ and each ball in the bottom row is labeled by an element of $\ZZ^+$;
\item paired balls have labels with the same absolute value;
 \item\label{GSMLQ:item2} a ball with label $a\in \ZZ^+$ in the top row
            must always have a ball labeled $a'$ underneath it, where $a' \geq a$, and if $a'=a$, the two balls must be trivially paired;
    \item\label{GSMLQ:item3} In the top row, each negative ball with label $-a$ (for $a\in \ZZ^+$) has either an empty spot below it or a ball with label $a'$, where $a \geq a'$.
\end{enumerate}
\end{definition} 
The balls in the bottom row of a signed two-line queue are called \emph{regular balls}, and those in the top row are called \emph{signed balls}.
The bottom row is then represented by a composition $\mu\in \NN^n$, and the top row by a signed permutation $\alpha$ of $\mu$. Let $\mcG^\alpha_\mu$ denote the
set of signed two-line queues with bottom row $\mu$ and top row $\alpha$. 

See \Cref{fig:ghost_two_line_queue} for an example of a signed two-line queue.

\begin{remark}\label{rem:sMLQ}
As in \cref{rem:MLQ}, a \emph{signed multiline queue}
    is a vertical concatenation of 
    two-line queues, which alternate between signed and (usual) two-line queues, with rows labeled $1, 1', 2, 2', \dots, L, L'$ from bottom to top.
    We refer to a collection of linked pairs of balls as a \emph{strand},
    and we require that all balls in a strand whose top ball is in row $r$, have labeled which is $\pm r$.
    See \cite{BenDaliWilliams2025} for the precise definition.
\end{remark}

\begin{figure}[t]
    \centering
    \includegraphics[width=0.5\linewidth]{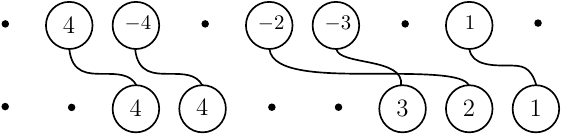}
    \caption{An example of a signed two-line queue in $\mcG_{(0,0,4,4,0,0,3,2,1)}^{(0,4,-4,0,-2,-3,0,1,0)}.$
    }
    \label{fig:ghost_two_line_queue}
\end{figure}

We associate to each pairing $p$ in a signed two-line queue a weight $\wt(p)$ defined as follows: we read the balls in the top row in decreasing order of the absolute value of their label; within a fixed absolute value, we read the balls from right to left.  Reading the balls in this order, we place the strands pairing the balls one by one.  The balls in the bottom row that have not yet been matched  are \emph{free}.  If a pairing $p$ matches a ball labeled $\pm a$ in the top row and column $j$ to a ball labeled $a$ in the bottom row and column $k>j$, then the free balls (respectively, empty positions) in the bottom row and columns $j+1,j+2,\dots,k-1$
are  \emph{skipped} (respectively, \emph{empty}).  We then set 
\begin{equation}\label{eq:pair2}
\wt(p) = \begin{cases}
         (1-t) t^{\skipped(p)+\emp(p)} &\text{ if $p$ connects a positive ball and a regular ball}\\
         -(1-t) t^{\skipped(p)+\emp(p)} &\text{ if $p$ connects a negative ball and a regular ball.}
\end{cases}
\end{equation}
Having associated a weight to each nontrivial pairing,
we define $$\wt_{\pair}(Q) = \prod_p \wt(p),$$
where the product is over all nontrivial pairings of balls in $Q$.
We define the weight generating function $b_\mu^\alpha$ of $\mcG^\alpha_\mu$ to be
\begin{equation}\label{eq:generating_function1}
    b_\mu^\alpha=b_\mu^\alpha(t):=\sum_{Q\in\mcG_\mu^\alpha}\wt_{\pair}(Q).
\end{equation}

\begin{remark}
In this paper, we are interested in the case when $\lambda$ has only one vacancy i.e $m_0(\lambda)=1$. In this case, one can check that for any pairing $p$ we have $\emp(p)=0$, and that $\skipped(p)$ has the following simpler description: if  $p$ matches a ball labeled $\pm a$ in the top row and column $j$ to a ball labeled $a$ in the bottom row and column $k>j$, then $\skipped(p)$ corresponds to the number of balls in the bottom row and in a column $j+1,j+2,\dots,k-1$ which are labeled $b<a$.
    
\end{remark}

We assign to each ball $B$ in the top row in column $k$ of a signed two-line queue $Q$ a weight
$$\wt(B):=\begin{cases}
    x_k &\text{if $B$ is positive}\\
    -\frac{1}{t^{n-1}}&\text{if $B$ is negative}.
\end{cases}$$

The total weight of $Q$ is then defined by 
$$\wt(Q):=\wt_{\pair}(Q)\prod_{B \text{ in the top row}}\wt(B).$$

Adding the ball weights to \cref{eq:generating_function1} we obtain 
\begin{equation}\label{eq:generating_function2}
    \wt_\alpha b_\mu^\alpha=\sum_{Q\in\mcG_\mu^\alpha}\wt(Q),
\end{equation}
with 
\begin{equation}
  \wt_\alpha:=\prod_{k:\,\alpha_k>0}x_k  \prod_{k:\, \alpha_k<0}\frac{-1}{t^{n-1}}.
\end{equation}

\subsection{The case when \texorpdfstring{$\lambda$}{lambda} has distinct parts.}\label{ssec:MLQ_distinct_parts}
In this section we assume that all parts of $\lambda$ are distinct ($m_i(\lambda)\leq 1$). We give an unsigned version of multiline queues in this case, which we use to prove that the generating functions $\wt_\alpha b_\mu^\alpha$ is actually positive under some assumptions on the parameters $t, x_1,\dots, x_n$. This property will be useful to obtain a Markov chain with positive transition rates.

\begin{definition}\label{def:Gbar}
Given a signed two-line queue $Q\in \mcG^\alpha_{\mu}$, we associate to it an unsigned version $\Qbar$ obtained by forgetting the signs of the balls in the top row. The composition we read in the bottom row (respectively the top row) of $\Qbar$ is $\mu$ (respectively $\lVert \alpha\rVert)$, where 
$$\lVert \alpha\rVert=(|\alpha_1|,\dots,|\alpha_n|).$$
We then define $\Gbar_\mu^{\kappa}$ as the set of paired ball systems obtained by applying this operation on $Q\in \mcG^\alpha_\mu$, where $\alpha\in \ZZ^n$ satisfying $\lVert\alpha\rVert=\kappa$.
\end{definition}

This leads us to define the following weights. Fix $\Qbar\in \Gbar^{\kappa}_\mu$:
\begin{itemize}
    \item A nontrivial pairing $p$ in $\Qbar$ has the weight
    \begin{equation}\label{eq:pairing_weight_unsigned}
  \wt(p)=(1-t)t^{\skipped(p)}.      
    \end{equation}

    \item Let $B$ be a ball labeled $a>0$ in column $k$ and such that the ball below is labeled $b$ (If $B$ has a vacancy below it, we take $b=0$.) We define the weight of $B$ by:
\begin{equation}\label{eq:weights_unsigned_balls}
\wt(B):=
\begin{cases}
    x_k-\frac{1}{t^{n-1}}&\text{if $b=a$,}\\
    x_k&\text{if $b>a$,}\\
    \frac{1}{t^{n-1}}&\text{if $b<a$}.
\end{cases}    
\end{equation}    
The weight of $\Qbar$ is defined by
$$\wt(\Qbar):=\prod_{B \text{ in the top row}}\wt(B) \prod_{p \text{ nontrivial pairing}}\wt(p).$$
\end{itemize}
 We then have the following lemma.

 \begin{lem}\label{lem:forget_signs}
 Fix a partition $\lambda$ with distinct part and two compositions $\kappa,\mu\in S_n(\lambda)$. Let $\Qbar\in \Gbar^{\kappa}_\mu$. Then
 $$\wt(\Qbar)=\sum_{Q}\wt(Q),$$
where the sum is taken over all signed two-line queues $Q$ from which $\Qbar$ is obtained by forgetting signs.
 \end{lem}
\begin{proof}
    We consider all the possible ways of ``adding signs'' to the balls in the top row of $\Qbar$ to obtain a signed two-line queue. Fix such a ball $B$ labeled $a>0$:
    \begin{itemize}
        \item if $B$ has a vacancy below it or a ball labeled $b<a$, then from \cref{def:signed_two-row} \cref{GSMLQ:item2}, we should assign a $-$ sign to $B$. 
        \item if $B$ has a ball labeled $b>a$ below it, then from \cref{def:signed_two-row} \cref{GSMLQ:item3}, we should assign a $+$ sign to $B$.
        \item if $B$ has a ball labeled $b=a$ below it, then from \cref{def:signed_two-row} \cref{GSMLQ:item1} and the fact that $\lambda$ has distinct parts we know that the two balls should be trivially paired. In this case we can give $B$ a $+$ or $-$ sign.
    \end{itemize}
We then check that the possible signs for each ball $B$ is consistent with the choice of weights in \cref{eq:weights_unsigned_balls}. In particular, one notices that when a ball $B$ is given a $-$ sign, the ball weight should be multiplied by $-1$ when we go from $\Qbar$ to $Q$, but the weight of the pairing connected to $B$ is also multiplied by $-1$ (compare \cref{eq:pair2,eq:pairing_weight_unsigned}).
\end{proof}

Given $\kappa\in S_n(\nu)$, we define $\tb_\nu^{\kappa}$
\begin{equation}\label{eq:def_c}
  \tb_{\nu}^{\kappa}:=\sum_{\alpha:\, \lVert \alpha\rVert =\kappa}\wt_\alpha b^{\alpha}_{\nu}.  
\end{equation}

 We get the following corollary obtained by combining \cref{eq:generating_function2} and \cref{lem:forget_signs}.
\begin{lem}\label{lem:c}
 Fix $\lambda$ a partition with distinct parts, and $\kappa,\mu\in S_n(\lambda)$. Then
 $$c^\kappa_\mu=\sum_{\Qbar\in \Gbar^{\kappa}_\mu} \wt(\Qbar).$$
\end{lem}
    One may notice that since $\lambda$ has distinct parts, $\Gbar^{\kappa}_\nu$ is either empty or contains exactly one element.
\begin{example}
 We give in \cref{fig:UGTLQ} the unique element of $\Gbar^{(2,7,1,5,0,6)}_{(0,7,2,1,5,6)}.$   The set $\Gbar^{(2,7,1,5,0,6)}_{(0,7,2,1
 ,6,5)}$ is however empty.
\end{example}
\begin{figure}
    \centering
    \includegraphics[width=0.5\linewidth]{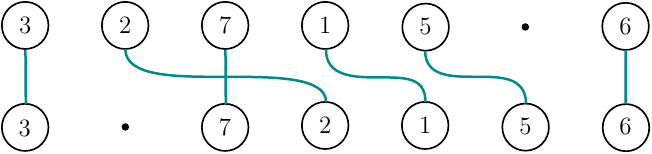}
    \caption{The unique paired ball system in $\Gbar^{(2,7,1,5,0,6)}_{(0,7,2,1,5,6)}.$}
    \label{fig:UGTLQ}
\end{figure}

Given a composition $\eta$, let $\eta^-:=(\eta^-_1,\dots,\eta^-_n)$,
where
$\eta^-_i=\max(\eta_i-1,0)$.

The weight generating functions of classical and signed two-line queues can be used to construct the interpolation ASEP polynomials. We state here this result for $q=1$; see \cite[Theorem 1.15 and Lemma 5.6]{BenDaliWilliams2025}.
     \begin{thm}[{\cite{BenDaliWilliams2025}}]\label{lem:F_decomposition}
 For any  composition $\nu\in \NN^n$, we have
 \begin{equation}\label{eq:F_decomposition}
   F^*_\nu(\bfx;1,t)=\sum_{\eta\in\NN^n}{F^{*\eta}_\nu}(\bfx;1,t) F^*_{\eta^{-}}(\bfx;1,t),  
 \end{equation}
 where 
 $$F^{*\eta}_{\nu} :=\sum_{\alpha\in\ZZ^n}b_\nu^\alpha \wt_\alpha a^\eta_{\lVert \alpha\rVert}=\sum_{\mu\in \NN^n}a_\mu^\eta c^\mu_\nu .$$
 \end{thm}

\section{Recoloring the interpolation \texorpdfstring{$t$}{t}-Push TASEP}\label{sec:recoloring_Push_TASEP}
Following \cite{AyyerMartinWilliams2025}, we introduce the following definition.
\begin{definition}
    We say that a function $\phi$ from $\NN$ to $\NN$ is \emph{weakly order-preserving} if $\phi(i)\leq \phi(j)$ whenever $i\leq j$.
\end{definition}
If $\mu\in \NN^n$ is a composition, we define $\phi(\mu):=(\phi(\mu_1),\dots,\phi(\mu_n))\in \NN^n$.
Note that if $\mu$ is a partition then so is $\phi(\mu)$. In what follows, we will often apply a weakly order-preserving map to states of the interpolation $t$-Push TASEP; we refer to this as 
\emph{recoloring}.

\begin{prop}\label{prop:lumping}
    Let $\phi:\NN \to \NN$ be a weakly order-preserving function with $\phi(0)=0$, and let $\lambda=(\lambda_1,\dots,\lambda_n)$ and $\kappa=(\kappa_1,\dots,\kappa_n)$ be  partitions such that $\kappa=\phi(\lambda)$.
Then the interpolation $t$-Push TASEP with particle content $\lambda$ lumps (or projects) to the interpolation $t$-Push TASEP with particle content $\kappa$.  That is,  
    if $\eta,\zeta\in S_n(\kappa)$ and $\mu\in S_n(\lambda)$ such that $\phi(\mu)=\eta$, then
    $$\sum_{\nu\in S_n(\lambda):\, \phi(\nu)=\zeta}\PP_{\lambda}(\mu,\nu)=\PP_\kappa(\eta,\zeta).$$
\end{prop}

\begin{proof}
We note that each of the three quantities $(P_j,\PPone{},\PPtwo{})$ defining the transition rates in the interpolation $t$-Push TASEP behaves well under recoloring: $P_j$ is independent from the state $\mu$, and for any $j$, we have:
\begin{align*}
  &\sum_{\rho\in S_n(\lambda):\, \phi(\rho)=\theta}\PPone{\lambda,j}(\mu,\rho)
  =\PPone{\kappa,j}(\eta,\theta),\quad \text{ for any $\eta,\theta\in S_n(\kappa)$ and $\mu\in S_n(\lambda)$ with $\phi(\mu)=\eta$,}\\
    &\sum_{\nu\in S_n(\lambda):\, \phi(\nu)
    =\zeta}\PPtwo{\lambda,j}(\rho,\nu)=\PPone{\kappa,j}(\theta,\zeta),\quad \text{ for any $\theta,\zeta\in S_n(\kappa)$ and $\rho\in S_n(\lambda)$  with $\phi(\rho)=\theta$.}
\end{align*}

 \cref{prop:lumping} follows from the above observations, \cref{rmq:twin_particles}, and the fact that the 
 state space of the first Markov chain projects onto the state space of the second.
\end{proof}

\begin{figure}[t]
    \centering
    \includegraphics[height=0.25\linewidth]{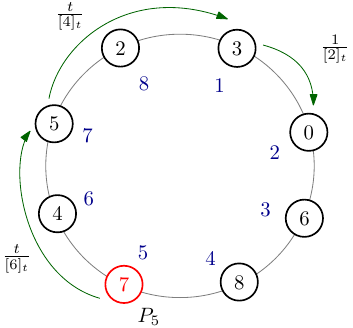}
    \includegraphics[height=0.26\linewidth]{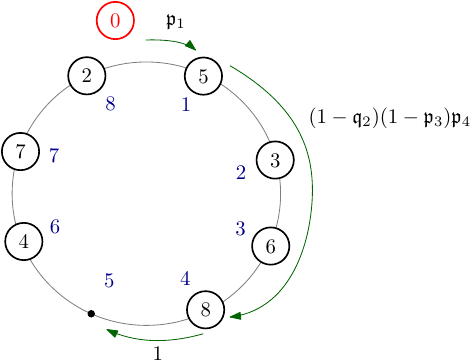}
    \includegraphics[height=0.23\linewidth]{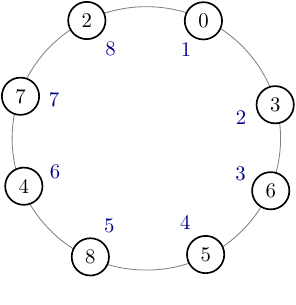}
    \caption{An example of the interpolation $t$-Push TASEP dynamics with $n=8$ for a restricted partition $\lambda=(8,7,6,5,4,3,2,0)$ and a state $\mu=(3,0,6,8,7,4,5,2)\in S_n(\lambda)$.}
    \label{fig:Markov_chain_2}
\end{figure}

\begin{example}
    In \cref{fig:Markov_chain_2}, we illustrate a transition for the $t$-Push$^*$ TASEP indexed by $\lambda=(8,7,6,5,4,3,2,0)$.  The corresponding transition for the $t$-Push$^*$ TASEP indexed by indexed by $\kappa=(4,3,3,2,1,1,0,0)$ from \cref{fig:Markov_chain} is then obtained by projecting, using for example by the map:
    $$\phi(i)=\begin{cases}
        0 & \text{if $i\in \{0,1,2\}$}\\
        1 & \text{if $i\in \{3,4\}$}\\
        2 & \text{if $i\in \{5\}$}\\
        3 & \text{if $i\in \{6,7\}$}\\
        4 & \text{if $i\geq 8$.}\\
    \end{cases}$$
\end{example}

We then have the following consequence on the stationary distribution. It follows from the strong lumping property of Markov chains (see \textit{e.g} \cite{KemenySnell1976,Pang2019,Williams2022}.)

\begin{cor}\label{cor:lumping}
Let $\phi,\lambda$ and $\kappa$ as above. Then  the stationary distributions of the 
interpolation $t$-Push TASEPs with particle contents $\kappa$ and $\lambda$ are related as follows:
for any $\eta\in S_n(\kappa)$, 
$$\pi^*_\kappa(\eta)=\sum_{\rho\in S_n(\lambda):\, \phi(\rho)=\eta}\pi^*_\lambda(\rho).$$
\end{cor}

\section{Interpolation polynomials}\label{sec:interpolation_polynomials}

\subsection{Some properties of interpolation polynomials}
\begin{definition}\label{def:int_elem}
Fix $n\geq 1$. For any $k\geq 1$, we define the \emph{elementary interpolation polynomials}
$$e^*_{k}(x_1,\dots,x_n;t):=\sum_{\begin{subarray}{c}S\subseteq \llbracket n \rrbracket\\  \#S=k\end{subarray}}\prod_{i\in S}\left(x_i-\frac{t^{\#S^c\cap \llbracket i-1\rrbracket}}{t^{n-1}}\right),$$
where the sum runs over subsets $S$ of $\llbracket n\rrbracket:=\{1,\dots, n\}$ of size $k$, and where $S^c$ denotes the complement of $S$ in $\llbracket n\rrbracket$. By convention, we set
$$e_0^*(x_1,\dots,x_n;t)=1.$$
We extend this definition to any partition $\lambda=(\lambda_1,\dots,\lambda_n)$:
$$e_\lambda^*(\bfx;t):=\prod_{i=1}^n e^*_{\lambda_i}(\bfx;t).$$
\end{definition}
The top homogeneous part of $e_{\lambda}^*$ correspond to the elementary symmetric polynomials.
We recall the following results for the interpolation ASEP polynomials; see \cite[Lemma 7.2]{BenDaliWilliams2025}.
\begin{lem}[\cite{{BenDaliWilliams2025}}]\label{lem:P_column}
        For any $\mu\in\{0,1\}^n$, we have
        \begin{equation}\label{eq:factorization_one_row}
          F_\mu^*(x_1,\dots,x_n;q,t)=\prod_{i\in S_\mu}\left(x_i-\frac{t^{\#S^c_\mu\cap \llbracket i-1\rrbracket}}{t^{n-1}}\right), \end{equation}
          where $S_\mu:=\{i\in \llbracket n\rrbracket: \, \mu_i=1\}.$
        As a consequence, 
    \begin{equation}\label{eq:P_column}
          P^*_{(1^m,0^{n-m})}(x_1,\dots,x_n;q,t)=\sum_{\mu\in S_n(1^m,0^{n-m})}F^*_{\mu}(x_1,\dots,x_n;q,t)=e^*_m(x_1,\dots,x_n;t).  
        \end{equation}
        The top homogeneous part of this equation gives
    \begin{equation}\label{eq:P_column_hom}
          P_{(1^m,0^{n-m})}(x_1,\dots,x_n;q,t)=\sum_{\mu\in S_n(1^m,0^{n-m})}F_{\mu}(x_1,\dots,x_n;q,t)=e_m(x_1,\dots,x_n;t).  
        \end{equation}
\end{lem}

We have the following factorization formula for interpolation symmetric Macdonald polynomials at $q=1$.
\begin{thm}[\cite{Dolega2017,BenDaliWilliams2025}]\label{thm:factorization_property}\label{thm:factorization}
For any partition $\lambda$
    \begin{equation}\label{eq:factorization}
      P^*_\lambda(x_1,\dots,x_n;1,t)=\prod_{1\leq i\leq \lambda_1}P^*_{\lambda'_i}(x_1,\dots,x_n;1,t)=\prod_{1\leq i\leq \lambda_1}e^*_{\lambda'_i}(x_1,\dots,x_n;t),
    \end{equation}
    where $\lambda'$ is the partition conjugate to $\lambda$.
\end{thm}

\subsection{The weak reordering property}
Fix a weakly order-preserving function $\phi:\NN\rightarrow\NN$.  Fix two partitions $\lambda$ and $\kappa$ such that $\phi(\lambda)=\kappa$.
For $\eta\in S_n(\kappa)$, define 
$$G^*_\eta(\bfx;t):=\sum_{\rho\in S_n(\lambda):\, \phi(\rho)=\eta} F^*_\rho(\bfx;1,t).$$ Let $G_{\eta}$ be the top homogeneous part of $G^*_\eta$.

The following is an analogue of \cite[Theorem 4.18]{AyyerMartinWilliams2025}. 
\begin{thm}\label{thm:weak-reordering}
    Fix $\lambda$ and $\kappa$ as above. For all $\eta\in S_n(\kappa)$, we have at $q=1$
    $$\frac{G^*_\eta(\bfx;t)}{P^*_\lambda(\bfx;1,t)}=\frac{F^*_\eta(\bfx;1,t)}{P^*_\kappa(\bfx;1,t)}.$$
\end{thm}
The proof of this theorem is very similar to the one given in \cite{AyyerMartinWilliams2025} in the homogeneous case, using results from \cite{AlexanderssonSawhney2019} on {nonsymmetric  Macdonald polynomials}. For completeness, we give a proof of this result in \cref{app:shape_permuting,app:permuted_basement}.

\begin{cor}\label{cor:eta_minus}
    Consider a composition $\rho$ with $\rho_i\neq 1$ for any $1\leq i\leq n$. Let $k$ be the number of non-zero parts of $\rho$. Set $\eta=\rho^-$.
    We then have at $q=1$,
$$F^*_{\rho}(\bfx;1,t)=F^*_{\eta}(\bfx;1,t)\cdot e^*_{k}(\bfx;t).$$
\end{cor}
\begin{proof}
Let $\lambda$ and $\kappa$ be the two partitions obtained by reordering $\rho$ and $\eta$, respectively.
  Consider the weakly order-preserving function $\phi:i\mapsto \max(i-1,0)$. We then have $\phi(\rho)=\eta$. Since $\lambda$ does not have parts of size 1, and $\phi$ is bijective from $\{0,2,3,\dots\}$ to $\{0,1,2,\dots\}$, then $\rho$ is the unique composition in $S_n(\lambda)$ such that $\phi(\rho)=\eta$ and we have $G^*_\eta=F^*_\rho$. It follows then from \cref{thm:weak-reordering} that
  $$\frac{F^*_{\rho}(\bfx;1,t)}{P^*_{\lambda}(\bfx;1,t)}=\frac{F^*_{\eta}(\bfx;1,t)}{P^*_{\kappa}(\bfx;1,t)}.$$
We then use the factorization property of the interpolation symmetric polynomials (\cref{thm:factorization_property}) and the fact that $\kappa$ is obtained from $\lambda$ by removing the largest column (of size $k$), to get that
$$\frac{P^*_{\lambda}(\bfx;1,t)}{P^*_{\kappa}(\bfx;1,t)}=e^*_{k}(\bfx;t)$$
which implies that $F^*_{\rho}(\bfx;1,t)=F^*_{\eta}(\bfx;1,t)\cdot e^*_{k}(\bfx;t).$
\end{proof}

\section{The proof of the main theorem}\label{sec:proof}
\subsection{The proof of \texorpdfstring{\cref{thm:main}}{the main theorem} for restricted partitions}

\begin{definition}\label{def:restricted}
    A \emph{restricted composition}
    (respectively, a \emph{restricted partition}) $\mu=(\mu_1,\dots,\mu_n)$ is a composition (respectively, partition)
    which has all parts distinct, a unique part of size $0$, and no part of size $1$.
\end{definition}

\emph{In this section we will assume that each $\lambda$ is a restricted partition as in \cref{def:restricted}.}
\subsubsection{The interpolation \texorpdfstring{$t$}{t}-Push TASEP when $\lambda$ is restricted}

The following lemma relates the transition probabilities of the classical $t$-Push TASEP (or equivalently \ref{Step1} in \cref{def:intpushTASEP}) to the generating function of two-line queues.

\begin{prop}[{\cite[Lemma 5.4]{AyyerMartinWilliams2025}}]\label{prop:Pone}
Fix $\mu,\rho\in S_n(\lambda)$, and let $j$ be the index such that $\rho_j=0$.  Then the following quantities are equal:
\begin{itemize}
\item the probability that in 
Step 1 of \cref{def:intpushTASEP} (i.e. the $t$-Push TASEP) we transition from $\mu$ to $\rho$ when the bell rings at site $j$ in $\mu$,
\item the weight generating function of classical
two-line queues with top row $\mu$ and bottom row $\rho$.
\end{itemize}
That is,
$$\PPone{_j}(\mu,\rho)=a^\mu_{\rho}.$$
\end{prop}

To help prepare the reader for the proof of \cref{prop:Ptwo}, we 
will prove \cref{prop:Pone} here.  The main idea  is that
a two-line queue encodes a transition of the $t$-Push TASEP.  For instance, \cref{fig:two_line_queue_restricted} encodes 
Step 1 in the transition illustrated in \cref{fig:Markov_chain_2}:
the top row encodes the initial state, each nontrivial pairing encodes a particle which moves, and 
the bottom row encodes the state after Step 1.

\begin{figure}[t]
    \centering
    \includegraphics[width=0.5\linewidth]{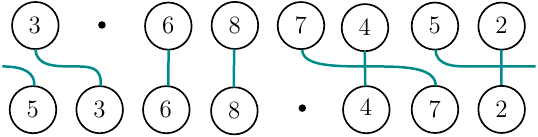}
    \caption{The (classical) two-line queue encoding Step 1 in the transition of \cref{fig:Markov_chain_2}.}
    \label{fig:two_line_queue_restricted}
\end{figure}

\begin{proof}
Note that since by assumption $\mu$ and $\rho$ are restricted compositions, they uniquely determine the pairings of the balls in a two-line queue whose top row is $\mu$ and bottom row is $\rho$.  Thus, there is at most one two-line queue with top row $\mu$ and bottom row $\rho$.  There exists one such queue if and only if when we pair the balls with the same labels, the bottom ball in a nontrivial pairing always has a ball with a smaller label directly above it.
 This condition is equivalent to the statement (*) that if we write the composition $\rho$ underneath the composition $\mu$, then for each position $i\neq j$ such that $\mu_i \neq \rho_i$, we have $\mu_i < \rho_i$. Moreover, when we have such a two-line queue,
its weight is the product of the pairing weights, each of which have the form 
$$\frac{(1-t)t^{\skipped(p)}}{1-t^{\free(p)}} = \frac{t^{\skipped(p)}}{[\free(p)]_t}.$$

Meanwhile, there is a unique series of ball displacements 
that must occur in order to realize a transition
in the $t$-Push TASEP from $\mu$ to $\rho$: namely, 
if $b>b^{(1)} > \dots > b^{(\ell)}$ are the letters in 
$\mu$ which appear in different locations in $\rho$, 
where $b$ is the ball at location $j$ in $\mu$, then 
first the ball $b$ must travel to the location of ball
$b^{(1)}$ (settling in that position and displacing 
the ball labeled $b^{(1)}$); then the ball $b^{(1)}$
must travel to the location of ball $b^{(2)}$ (settling in that position and displacing the ball labeled $b^{(2)}$); and so on.  If the resulting configuration is $\rho$, then there is a nonzero probability of transitioning from $\mu$ to $\rho$. 
Note e.g. that if $j_1$ is the location of ball $b^{(1)}$ in $\mu$,
then we are saying that $\mu_{j_1} = b^{(1)}$ and $\rho_{j_1}=b > \mu_{j_1}$ (and similarly for subsequent balls), which is equivalent to the statement (*) above.  Thus, we can construct a two-line queue whose pairings encode the sequence of ball displacements.
Moreover, note that in Step 1 of \cref{def:intpushTASEP},
the probability of ball $b$ displacing ball $b^{(1)}$ 
is $\frac{t^{k-1}}{[m]_t}$, when there are $m$ weaker particles and $b^{(1)}$ is the $k$th such particle that $b$ encounters when it travels clockwise; this exactly corresponds to the statistic 
$\frac{(1-t)t^{\skipped(p)}}{1-t^{\free(p)}} = \frac{t^{\skipped(p)}}{[\free(p)]_t}$ 
in the corresponding two-line queue.
\end{proof}

Recall that \cref{fig:two_line_queue_restricted} shows the two-line queue corresponding to Step 1 in the transition of \cref{fig:Markov_chain_2}, where $j=5$, $\mu=(3,0,6,8,7,4,5,2)$ and $\rho=(5,3,6,8,0,4,7,2)$. One then checks that 
$$\PPone{j}(\mu,\rho)=\frac{t}{[6]_t}\times\frac{t}{[4]_t}\times \frac{1}{[2]_t}=a^\mu_\rho.$$

\begin{remark}\label{rmk:subsets}
Each possible transition of the t-Push TASEP and each instance of Step 2 above can be indexed by 
choosing a subset of positions, where the corresponding particles have distinct labels.  For the $t$-Push TASEP, once we have the subset of positions, we know that the particles have to get bumped in decreasing order of their labels.  And in Step 2, the particles get bumped from left to right.
\end{remark}

We now give an analogous result relating the transition probabilities of \ref{Step2} of \cref{def:intpushTASEP}  to the generating function for signed two-line queues.

\begin{prop}\label{prop:Ptwo}
     Fix $\rho,\nu\in S_n(\lambda)$, and let $j$ be the index such that $\rho_j=0$. We have 
     $$\PPtwo{j}(\rho,\nu)=\frac{c^\rho_\nu}{\prod_{k<j}\left(x_k-\frac{1}{t^{n-2}}\right)\prod_{k>j}\left(x_k-\frac{1}{t^{n-1}}\right)},$$
     or equivalently,
     $$ P_j\cdotp \PPtwo{j}(\rho,\nu)=\frac{c^\rho_\nu}{e_{n-1}^*},$$
     where $c^\rho_\nu$ is the coefficient from \cref{eq:def_c} and \cref{lem:c}, i.e. 
     the generating function for the set $\Gbar_{\nu}^\rho$.
\end{prop}

\begin{figure}
    \centering
    \includegraphics[width=0.5\linewidth]{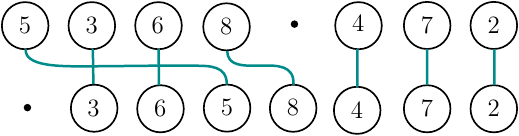}
    \caption{The element of $\Gbar^{(5,3,6,8,0,4,7,2)}_{(0,3,6,5,8,4,7,2)}$ (see \cref{def:Gbar})  encoding the Step 2 of the transition of \cref{fig:Markov_chain_2}.}
    \label{fig:unsigned_two_line_queue_restricted}
\end{figure}
As in the proof  of \cref{prop:Pone}, the main idea is that a signed two-line queue encodes Step 2 of the 
interpolation $t$-Push TASEP.  For instance, 
\cref{fig:unsigned_two_line_queue_restricted} encodes Step 2 
of the transition of \cref{fig:Markov_chain_2}.

\begin{proof}
 Note that \ref{Step2} of \cref{def:intpushTASEP} is encoded by 
an element of a set $\Gbar^{\rho}_{\nu}$ (see \cref{def:Gbar}).  
Indeed, the transition in \ref{Step2} from the configuration $\rho$ to the configuration $\nu$ is possible if and only there is an
element $\Qbar$ in $\Gbar^{\rho}_\nu$ (recall that this set contains at most one element).
More precisely,  a particle labeled $a>0$ which moved from position $k\in \llbracket n\rrbracket$ to a position $k'$, corresponds to a non trivial pairing in $\Qbar$ connecting a ball labeled $a$ in column $k$ of the top row to a ball labeled $a$ in column $k'$ of the bottom row. Particles which do not move correspond to trivial pairings. 

We now claim that $\wt(\Qbar)$ divided by 
$D:=\prod_{k<j}\left(x_k-\frac{1}{t^{n-2}}\right)\prod_{k>j}\left(x_k-\frac{1}{t^{n-1}}\right)$ gives $\PPtwo{j}(\rho,\nu)$.  We will prove the claim below by showing that each 
ball or pairing weight in $\wt(\Qbar)$,
divided by one of the factors in $D$, equals one of the skipping/ displacement
probabilities from \eqref{eq:probskip} or \eqref{eq:prob} (whose product is $\PPtwo{j}(\rho,\nu)$).  Note that in what follows,
instead of associating the weight 
$(1-t)t^{\skipped(p)}$ to each nontrivial pairing,
we will associate $(1-t)$ to the top ball in each nontrivial pairing,
and a factor of $t$ to each skipped ball.
\begin{itemize}
    \item Each ball in column $k>j$ of $\Qbar$ is necessarily trivially paired, since no ball in position $k>j$ get skipped or displaced in \ref{Step2}.
      In $\Qbar$ this ball gets  weight $x_k-\frac{1}{t^{n-1}}$; when we divide this weight by the $k$th factor of $D$, we get $1$, which corresponds to the fact that balls in position $k>j$ do not contribute to $\PPtwo{j}(\rho,\nu)$.
        \item A ball in $\Qbar$ labeled $b$ in column $k<j$ which is trivially paired, and which is not skipped by a ball $a>b$, also has weight $x_k-\frac{1}{t^{n-1}}$. When we divide this weight by the $k$th factor of $D$, we get $1-\mfp_k$ (see \eqref{eq:1minuspk}).  This is what we desired, because such a trivial pairing in $\Qbar$ corresponds to a particle labeled $b$ which is skipped over by a particle with a smaller label, and hence contributes $1-\mfp_k$
        to $\PPtwo{j}(\rho,\nu)$.
    \item A ball in $\Qbar$ labeled $b$ in column $k<j$ which is trivially paired, and which is skipped by a ball $a>b$, gets a weight $t(x_k-\frac{1}{t^{n-1}})$.  When we divide this weight by the 
    $k$th factor of $D$, we get $1-\mfq_k$
    (see \eqref{eq:1minuspk}).
    This is what we desired, because such a trivial pairing corresponds to a particle labeled $b$ skipped over by a particle with a larger label, and hence contributes $1-\mfq_k$ to $\PPtwo{j}(\rho,\nu)$.
    \item A ball labeled $b$ in the top row of $\Qbar$ in column $k<j$ which has a ball labeled $a<b$ below it gets a weight $(1-t)\frac{1}{t^{n-1}}$    (the factor $(1-t)$ is the nontrivial pairing weight). When we divide this weight by the $k$th factor of $D$, we get $\mfp_k$.  This is what we desired, because this pairing corresponds to a particle labeled $b$ being displaced by a particle with a smaller label, and hence contributing $\mfp_k$ 
    to $\PPtwo{j}(\rho,\nu)$.
    \item A ball labeled $b$ in the top row of 
    $\Qbar$ in column $k<j$ which has a ball labeled $a>b$ below it gets a weight $(1-t)x_k$ (the factor $(1-t)$ is the nontrivial pairing weight). 
     When we divide this weight by the $k$th factor of $D$, we get $\mfq_k$.  This is what we desired, because this pairing
    corresponds to a particle labeled $b$ being displaced by a particle with a larger label, and hence contributing $\mfq_k$ to $\PPtwo{j}(\rho,\nu)$.\qedhere
\end{itemize}
\end{proof}

Recall that \cref{fig:unsigned_two_line_queue_restricted} shows the paired ball system corresponding to Step 2 in the transition of \cref{fig:Markov_chain_2}, where $j=5$, $\rho=(5,3,6,8,0,4,7,2)$ and $\nu=(0,3,6,5,8,4,7,2)$. We then have 
\begin{align*}
  \PPtwo{5}(\rho,\nu)&=\mfp_1\cdotp(1-\mfq_2)\cdotp(1-\mfp_3)\mfp_4\\
  &=\frac{t^{-7}(1-t)}{x_1-t^{-6}}\cdotp \frac{t (x_2-t^{-7})}{x_2-t^{-6}}\cdotp\frac{x_3-t^{-7}}{x_3-t^{-6}}\cdotp\frac{t^{-7}(1-t)}{x_4-t^{-6}}.
\end{align*}
On the other hand, for the paired ball system $\Qbar$ of \cref{fig:unsigned_two_line_queue_restricted}, has ball weight (as given in \cref{eq:weights_unsigned_balls}) 
$$t^{-7}\cdotp (x_2-t^{-7})\cdotp(x_3-t^{-7})\cdotp t^{-7}\cdotp (x_6-t^{-7})\cdotp (x_7-t^{-7}) \cdotp(x_8-t^{-7})$$
and its pairing weight (defined in \cref{eq:pairing_weight_unsigned}) is 
$$(1-t)^2\cdotp t,$$
and $\wt(\Qbar)$ is the product of these two quantities. We then get that
$$\frac{c^\rho_\nu}{\prod_{k<5}\left(x_k-t^{-6}\right)\prod_{k>5}\left(x_k-t^{-7}\right)}
=\frac{\wt(\Qbar)}{\prod_{k<5}\left(x_k-t^{-6}\right)\prod_{k>5}\left(x_k-t^{-7}\right)}=\PPtwo{5}(\rho,\nu).$$

We then get the following proposition connecting the interpolation $t$-Push TASEP model to signed multiline queues.
\begin{prop}\label{prop:intTASEP_SMLQ}
If $\lambda$ is restricted, and $\mu,\nu\in S_n(\lambda)$, then    
\begin{align*}
    \PP(\mu,\nu)=\sum_{\rho\in S_n(\lambda)}  \frac{a^\mu_\rho c^\rho_{\nu}}{e^*_{n-1}}.
\end{align*}
\end{prop}
\begin{proof}
    Combining \cref{prop:Pone} and \cref{prop:Ptwo}, we get
    \begin{align*}
    \PP(\mu,\nu)
    &=\sum_{1\leq j\leq n}P_j\sum_{\rho\in S_n(\lambda):\rho_j=0}\PPone{j}(\mu,\rho)\PPtwo{j}(\rho,\nu)\\
    &=\sum_{1\leq j\leq n}\sum_{\rho\in S_n(\lambda):\rho_j=0}  \frac{a^\mu_\rho c^\rho_{\nu}}{e^*_{n-1}}\\
    &=\sum_{\rho\in S_n(\lambda)}  \frac{a^\mu_\rho c^\rho_{\nu}}{e^*_{n-1}}.\qedhere
\end{align*}
\end{proof}

\subsubsection{The case when \texorpdfstring{$\lambda$}{lambda} is restricted in \cref{thm:main}}
Fix a partition $\lambda$ with distinct parts, with one part of size 0 and no parts of size 1. 
Let $\nu\in S_n(\lambda)$.
From \cref{lem:F_decomposition}, we have
 \begin{equation*}
   F^*_\nu(\bfx;1,t)=\sum_{\eta\in\NN^n}{F^{*\eta}_\nu}(\bfx;t) F^*_{\eta^{-}}(\bfx;1,t), 
 \end{equation*}
 where 
 $$F^{*\eta}_{\nu}(\bfx;t) :=\sum_{\alpha\in\ZZ^n}b_\nu^\alpha \wt_\alpha a^\eta_{\lVert \alpha\rVert}=\sum_{\kappa\in \NN^n}a_\kappa^\eta c^\kappa_\nu .$$
But we know from \cref{cor:eta_minus} that
$$F^*_{\eta^-}(\bfx;1,t)=\frac{F^*_{\eta}(\bfx;1,t)}{e_{n-1}^*(\bfx;t)},$$
we use here the fact that $\eta$ has a unique part of size 0. Hence
\begin{equation*}
   F^*_\nu(\bfx;1,t)=\sum_{\eta\in\NN^n}F^*_{\eta}(\bfx;1,t)\sum_{\kappa\in \NN^n}\frac{a_\kappa^\eta c^\kappa_\nu}{e_{n-1}^*(\bfx;t)},
 \end{equation*}
which can be rewritten using the transition probabilities of the interpolation $t$-Push TASEP (\cref{prop:intTASEP_SMLQ}) we get
\begin{equation*}
   F^*_\nu(\bfx;1,t)=\sum_{\eta\in\NN^n}F^*_{\eta}(\bfx;1,t)\PP(\eta,\nu).
 \end{equation*}
This proves that $F^*_\mu(\bfx;1,t)$ are proportional to the stationary distribution of the interpolation $t$-Push TASEP $\pi^*_\lambda(\mu)$.
Finally, we use \cref{prop:symmetrization} to deduce that $\frac{F^*_\mu(\bfx;1,t)}{P^*_\lambda(\bfx;1,t)}=\pi^*_\lambda(\mu).$

\subsection{Proof of \texorpdfstring{\cref{thm:main}}{the main theorem} for arbitrary partitions}\label{sec:arbitrary}
\begin{proof}[Proof of \cref{thm:main}]
Fix a partition $\kappa=(\kappa_1,\dots,\kappa_n)$, and $\eta\in S_n(\kappa)$.  We want to prove that 
$$\pi_\kappa^*(\eta)=\frac{F^*_\eta(\bfx;1,t)}{P^*_\kappa(\bfx;1,t)}.$$
It is not hard to see that we can find 
a restricted partition $\lambda$ and a weakly order-preserving function $\phi$ such that $\phi(\lambda)=\kappa$.
By \cref{cor:lumping}, we know that
        $$\pi^*_\kappa(\eta)=\sum_{\rho\in S_n(\lambda):\, \phi(\rho)=\eta}\pi^*_{\lambda}(\rho).$$
Since \cref{thm:main} holds for any restricted partition $\lambda$, we get
$$\pi^*_\kappa(\eta)=\sum_{\rho\in S_n(\lambda):\, \phi(\rho)=\eta}\frac{F^*_\rho(\bfx;1,t)}{P^*_\lambda(\bfx;1,t)}.$$
Now from \cref{thm:weak-reordering}, we have
$$\sum_{\rho\in S_n(\lambda):\, \phi(\rho)=\eta}\frac{F^*_{\rho}(\bfx;1,t)}{P^*_\lambda(\bfx;1,t)}
=\frac{G^*_{\eta}(\bfx;t)}{P^*_\lambda(\bfx;1,t)}
=\frac{F^*_\eta(\bfx;1,t)}{P^*_\kappa(\bfx;1,t)}.$$
We then deduce the desired statement
\begin{equation*}
  \pi^*_\kappa(\eta)=\frac{F^*_\eta(\bfx;1,t)}{P^*_\kappa(\bfx;1,t)}.\qedhere  
\end{equation*}
\end{proof}

\section{Density formulas}\label{sec:density}

The following proposition is a consequence of
\cref{thm:main} and  \cref{lem:P_column}.
\begin{prop}\label{prop:01}
Let $\lambda = (1^{m_1}, 0^{m_0})$ be a partition with $n$ parts whose parts are all $1$'s or $0$'s.  The stationary probability $\pi_\lambda^*(\mu)$ of $\mu \in S_n(\lambda)$ for the 
interpolation $t$-Push TASEP is 
$$\pi^*_\lambda(\mu) = 
\frac{F_{\mu}^*(x_1,\dots,x_n; q, t)}{e^*_{m_1}(x_1,\dots,x_n;t)} =
\frac{1}{e^*_{m_1}(x_1,\dots,x_n;t)} \prod_{i\in S_\mu}\left(x_i-\frac{t^{\#S^c_\mu\cap \llbracket i-1\rrbracket}}{t^{n-1}}\right), $$
    where $S_{\mu}:=\{i: \mu_i=1\}.$
\end{prop}
Recall that the \emph{density} at site $i$ is the probability of finding a particle at site $i$ in the stationary distribution.
We  use angle brackets 
$\langle \cdot \rangle$ to denote expectations in the stationary distribution.  Let 
$\eta_j$ be the indicator variable for a particle at site $j$, so that 
$\langle \eta_j\rangle$ is the density at site $j$.  

From the definition of density, we obtain the following formula. 
\begin{lem}\label{lem:density}
Let $\lambda = (1^{m_1}, 0^{m_0})$ be a partition with $n$ parts whose parts are all $1$'s or $0$'s.
  The density at state $j$ in the stationary distribution is given by 
  $$\langle \eta_j \rangle =
  \frac{\sum F^*_{\mu}(x_1,\dots,x_n; q,t)}{e^*_{m_1}(x_1,\dots,x_n;t)},$$
  where the sum is over all $\mu\in S_n(\lambda)$ which have a $1$ in site $j$.
\end{lem}
Using \cref{prop:01} and \cref{lem:density}, we obtain the following.
\begin{cor}\label{cor:density}
    Let $\lambda = (1^{m_1}, 0^{m_0})$ be a partition with $n$ parts whose parts are all $1$'s or $0$'s.
  The density at state $j=1$ in the stationary distribution is given by 
  \begin{equation}\label{eq:density}
  \langle \eta_1 \rangle =
  \frac{(x_1-\frac{1}{t^{n-1}}) e^*_{m_1-1}(tx_2, tx_3,\dots,tx_n;t)} {t^{m_1-1} e^*_{m_1}(x_1,\dots,x_n;t)},
  \end{equation}
   and the density at state $j=n$ in the stationary distribution is given by 
  $$\langle \eta_n \rangle =
  \frac{(x_n-\frac{t^{m_0}}{t^{n-1}}) e^*_{m_1-1}(tx_1, tx_2, \dots,tx_{n-1};t)} {t^{m_1-1} e^*_{m_1}(x_1,\dots,x_n;t)},$$
\end{cor}
One can actually write an explicit expression for $\langle\eta_j\rangle$ for any $1\leq j\leq n$. Here we only give the expressions for $j=1$ and $j=n$, for which the formulas are slightly simpler.

We will use the notation $\bfx:=(x_1,\dots,x_n)$, $\bfx^{(1)}:=(x_2,\dots,x_n)$,  $t\bfx:=(tx_1,\dots,tx_n)$ and so on.
\begin{definition}\label{def:sstar}
For any partition $\lambda=(\lambda_1,\dots,\lambda_n)$, we define the \emph{$t$-interpolation Schur polynomial $s_\lambda^*(x_1,\dots,x_n;t)=s^*_\lambda(\bfx;t)$} by
    $$s^*_\lambda(\bfx;t)=P^*_\lambda(\bfx;t,t).$$
\end{definition}

\cref{prop:density} below is an analogue of \cite[Corollary 7.1]{AyyerMartinWilliams2025}.  To prove it, we will use the following dehomogenization of the dual Jacob–Trudi identity, which is proved in \cref{app:t_int_Schur}.
\begin{lem}\label{lem:2column}
Let $\lambda=(2^a,1^b)$.  Then we have 
\begin{align*}
s^*_{\lambda}(t\bfx;t)&=e^*_{a+b,n-1}(t\bfx;t)e^*_{a,n-2}(t\bfx;t)-e^*_{a+b+1,n-2}(t\bfx;t)e^*_{a-1,n-1}(t\bfx;t)\\
&=t^a e^*_{a+b}(t\bfx;t) e^*_a(\bfx;t) - t^{a+b+1} e^*_{a+b+1}(\bfx;t) e^*_{a-1}(t\bfx;t).
\end{align*}
\end{lem}

\begin{prop}\label{prop:density}
Recall, as in \cref{not:type}, that  $M_i = m_i+m_{i+1} + \dots + m_L.$  Let $s^*_{\lambda}$ be the $t$-interpolation Schur polynomial from \cref{def:sstar}.
   The density of the particle of species $i$ in the first site in 
   the interpolation $t$-Push TASEP with content $\lambda=(L^{m_L}, \dots, 1^{m_1}, 0^{m_0})$ is given by
   \small{
   \begin{align*}
       \langle \eta_1^{(i)} \rangle =&
       \frac{(x_1-\frac{1}{t^{n-1}}) \left(e_{M_i-1}^*\left(t\bfx^{(1)};t\right) e^*_{M_{i+1}}\left(\bfx^{(1)};t\right)  - t^{m_i} e_{M_{i+1}-1}^*\left(t\bfx^{(1)};t\right) e_{M_i}^*\left(\bfx^{(1)};t\right) \right)}{t^{M_{i}-1} e^*_{M_i}\left(\bfx;t\right) e^*_{M_{i+1}}\left(\bfx;t\right)} \\
       =&
       \frac{(x_1-\frac{1}{t^{n-1}}) 
       \ s^*_{(2^{M_{i+1}},1^{m_i-1})}\left(t\bfx^{(1)};t\right)
       }{t^{M_{i+1}+M_i-1} e^*_{M_i}\left(\bfx;t\right)e^*_{M_{i+1}}\left(\bfx;t\right)}.
   \end{align*}}
\end{prop}

\begin{proof}
By \cref{prop:lumping} and the lumping property (\cref{cor:lumping}),  the density of the particle
    of species $i$ is the density of the particle of species $1$ in the single species
    Markov chain with $M_i$ particles, minus the density of the particle of species $1$
    in the single species Markov chain with $M_{i+1}$ particles.
    Thus using \cref{eq:density}, we have that 
     \begin{align*}
       \langle \eta_1^{(i)} \rangle &= 
       \frac{(x_1-\frac{1}{t^{n-1}}) e^*_{M_i-1}\left(t\bfx^{(1)};t\right)}{t^{M_i-1} e^*_{M_i}(\bfx;t)}  - \frac{(x_1-\frac{1}{t^{n-1}}) e^*_{M_{i+1}-1}\left(t\bfx^{(1)};t\right)}{t^{M_{i+1}-1} e^*_{M_{i+1}}(\bfx;t)} \\
       &= \frac{(x_1-\frac{1}{t^{n-1}})}{t^{M_{i}-1}} \cdot 
       \frac{e^*_{M_i-1}\left(t\bfx^{(1)};t\right) e^*_{M_{i+1}}(\bfx;t)-t^{m_i} e^*_{M_{i+1}-1}\left(t\bfx^{(1)};t\right) e^*_{M_i}(\bfx;t)}{e^*_{M_i}(\bfx;t) e^*_{M_{i+1}}(\bfx;t)}, \label{eq:sofar}
   \end{align*}
   where $\bfx:=(x_1,\dots,x_n)$.
If we now use the fact that 
\begin{equation}
e^*_d(\bfx;t) = \left(x_1-\frac{1}{t^{n-1}}\right) t^{-(d-1)} e^*_{d-1}\left(t\bfx^{(1)};t\right) + e^*_d\left(\bfx^{(1)};t\right),    
\end{equation}
we get the first statement of the proposition.  To get the second statement, we use \cref{lem:2column}.
\end{proof}

\begin{appendices}
    
\section{Shape permuting operators and Knop–Sahi recurrence}\label{app:shape_permuting}
The main purpose of \cref{app:shape_permuting} and \cref{app:permuted_basement} is to prove \cref{thm:weak-reordering}. We will follow the strategy of \cite{AlexanderssonSawhney2019} and \cite{AyyerMartinWilliams2025} (which proved the homogeneous case of \cref{thm:weak-reordering}). 

Fix $n\geq 1$. Let $\YY_n$ denote the set of integer partitions 
$\lambda=(\lambda_1,\dots,\lambda_n)=(\lambda_1 \geq \dots \geq \lambda_n)$ with at most $n$ parts.  We let $|\lambda|$ denote the sum $\lambda_1+ \dots + \lambda_n$ of the parts of the partition and call it the \emph{size} of $\lambda$.
Let $\mcP_n$ denote the ring of polynomials in $n$ variables, and 
let $\mcP_n^{(d)}$ denote the polynomials of degree at most $d$.

\subsection{The dehomogenization operator \texorpdfstring{$\psi$}{psi}}
For any partition $\lambda\in\YY_n$ of size $d$, we define the space $\VV\subset\mcP^{(d)}_n$ by
$$\VV:=\left\{f\in \mcP_n^{(d)}: f(\widebar\nu)=0 \text{ for any $|\nu|\leq |\lambda|$ and $\nu\notin S_n(\lambda)$}\right\},$$
where $\widebar{\nu}$ is the sequence defined in \cref{eq:k2}.

We recall that the \textit{nonsymmetric Macdonald polynomials} $E_{\mu} = E_\mu(\bfx;q,t)$, where $\mu$ ranges over all compositions with at most $n$ parts, form a basis for the space of polynomials $\mcP_n$; see e.g. \cite{HaglundHaimanLoehr2008}. 
We will also use the interpolation analogue of these polynomials $E^*_\mu(\bfx;q,t)$ \cite{Knop1997b,Sahi1996}, which can be characterized as follows. 
\begin{thm}[\cite{Knop1997b,Sahi1996}]\label{thm:Knop_Sahi}
    Fix $\mu\in \NN^n$ of size $d$. There exists a unique polynomial $E^*_\mu\in\mcP_n^{(d)}$, called the \emph{nonsymmetric interpolation Macdonald polynomial}, such that
    \begin{itemize}
        \item $[x^\mu]E^*_\mu=1$ (so in particular, $E^*_{\mu}$ has degree $d$),
        \item $E^*_\mu(\widetilde\nu)=0$ for any $\nu\in \NN^n$ satisfying $|\nu|\leq d$ and $\nu\neq \mu$.
    \end{itemize}
    Moreover, the top homogeneous part of $E^*_\mu$ is  $E_\mu$.
\end{thm}

When indexed by a partition $\lambda$, nonsymmetric interpolation Macdonald polynomials coincide with interpolation ASEP polynomials: $F^*_\lambda=E^*_\lambda$. We also have the following 
result from \cite[Lemma 2.5 and Proposition 2.14]{BenDaliWilliams2025}.
\begin{prop}[\cite{BenDaliWilliams2025}]\label{prop:Vstar}
    Fix  a partition $\lambda$. Then $\{E_{\mu}^* \ \vert \ \mu\in S_n(\lambda)\}$ and $\{F_{\mu}^* \ \vert \ \mu\in S_n(\lambda)\}$ are both bases for $\VV$. In particular, if $\mu\in S_n(\lambda)$, then $F^*_\mu$ is a linear combination of $E_{\nu}^*$ for $\nu\in S_n(\lambda)$.
\end{prop}
\begin{remark}\label{rmk:basis_1}
In the following, we will use the fact that even after evaluating at $q=1$, $\{E^*_\mu(\bfx;1,t) \ \vert \ {\mu\in \YY_n}\}$ and $\{F^*_\mu(\bfx;1,t) \ \vert \ {\mu\in \YY_n}\}$  are both bases of the polynomial space in $n$ variables and coefficients in $\QQ(q)$. Indeed, $E^*_\mu$ is triangular in the monomial basis (with respect to the dominance order on compositions) (see \cite[Theorem 3.11]{Knop1997b}) and the coefficient of $x_\mu$ in $E^*_\mu$ is 1, independent of $q$. A similar argument works for the polynomials $F^*_\mu$ using \cref{eq:normalization_F}.
\end{remark}

We will make use of the \emph{dehomogenization map}; see \cite[Definition 3.1 and Theorem 3.8]{NakviSahiSergel2023} for a similar definition in the context of Jack polynomials.
\begin{definition}\label{def:dehom}
We define a linear map $\boldsymbol{\psi}:\mcP_n \to \mcP_n$ by 
\begin{equation}
 \boldsymbol{\psi}: E_\mu\mapsto E_\mu^*, \quad\text{for any partition $\mu$.}   
\end{equation}
We call the operator $\boldsymbol{\psi}$ the \emph{dehomogenization map}.
\end{definition}

Since both $(E_\mu)$ and $(E_\mu^*)$ are bases of $\mcP_n$, this map is well defined and is an isomorphism.
Alternatively, $\boldsymbol{\psi}$ can be defined as the linear operator such that if $f$ is homogeneous of degree $d$ then $\boldsymbol{\psi}(f)$ is the unique polynomial of top homogeneous part $f$ and which vanishes on all compositions of size smaller than $d$. 

Using \cref{prop:Vstar}, and the fact that the top homogeneous part of $F^*_\mu$ is $F_\mu$, we get that
$\boldsymbol{\psi}(F_\mu)=F_\mu^*$. Thus we can conclude the following.

\begin{lem}\label{lem:dehomog}
If we have an expansion 
$E_\mu=\sum_{\nu}c_{\mu,\nu}F_\nu$
(for some coefficients $c_{\mu,\nu}\in \QQ(q,t)$), we get that 
$E^*_\mu=\sum_{\nu}c_{\mu,\nu}F^*_\nu$ by applying the map $\boldsymbol{\psi}$.
\end{lem}

\subsection{The Hecke operators}
We now recall the definition of the Hecke operators and their action on the ASEP polynomials (see e.g. \cite[Section 2.3]{BenDaliWilliams2025}).

For $1\leq i \leq n-1$, we let $s_i=(i,i+1)$ denote the transposition exchanging $i$ and $i+1$. The transposition $s_i$ acts on a polynomial in $\mcP_n$ by permuting the variables $x_i$ and $x_{i+1}$:
$$s_i\cdotp f(x_1,\dots,x_i,x_{i+1},\dots,x_n)
=f(x_1,\dots,x_{i+1},x_i,\dots,x_n).$$

The \emph{Hecke operator} $T_i$ is the linear operator on $\mcP_n$ defined by 
\begin{equation}\label{eq:Hecke}
T_i:=t-\frac{tx_i-x_{i+1}}{x_i-x_{i+1}}(1-s_i).
\end{equation}

These operators satisfy the relations of the Hecke algebra of type $A_{n-1}$
\begin{equation}\label{eq:Hecke_algebra}
\begin{array}{ll}
(T_i-t)(T_i+1)=0 & \text{for $1\leq i \leq n-1$}     \\
 T_iT_{i+1}T_i=T_{i+1}T_iT_{i+1}    & \text{for $1\leq i \leq n-2$}\\
 T_iT_j=T_jT_i    & \text{for $ |i-j| > 1$}.\\
\end{array}
\end{equation}

If $\sigma\in S_n$ and $\sigma=s_{i_1}\dots s_{i_\ell}$ is a reduced decomposition of $\sigma$, we define 
\begin{equation}\label{def:Hecke}
T_\sigma:=T_{i_1}\dots T_{i_\ell}.
\end{equation}
It follows from \eqref{eq:Hecke_algebra} that this definition is independent of the choice of reduced expression.

\begin{prop}\label{prop:T_i-F_star} \cite[Proposition 2.10]{BenDaliWilliams2025}
    Let $\mu = (\mu_1,\dots,\mu_n)$.  For $1\leq i \leq n-1$, the interpolation ASEP polynomials $F^*_\mu$ satisfy the following: 
    \begin{enumerate}
        \item\label{item 1} $T_i F^*_{\mu}=F^*_{s_i\mu}$ if $\mu_i>\mu_{i+1}$,
        \item\label{item 2} $T_iF^*_{\mu}=tF^*_{\mu}$ if $\mu_i=\mu_{i+1}$,
        \item\label{item 3} $T_i F^*_{\mu}=(t-1)F^*_{\mu}+tF^*_{s_i\mu}$ if $\mu_i<\mu_{i+1}$.
    \end{enumerate}
\end{prop}

Moreover, we have the following \emph{shape permuting
operator} formula for interpolation nonsymmetric Macdonald polynomials.

\begin{prop}[{\cite[Proposition 3.6]{BenDaliWilliams2025}}]
\label{prop:shape}
Let $\nu\in \NN^n$ be a composition and fix $1\leq j \leq n-1$. Suppose that $\nu_j>\nu_{j+1}$, and
write
\begin{equation}\label{eq:def_r}
r_j(\nu):=
\#\{k<j \mid \nu_{j+1}<\nu_{k}\leq \nu_j\}
+
\#\{k>j \mid \nu_{j+1}\leq \nu_k < \nu_j\}.
\end{equation}

Then
    \begin{equation}
\label{shape-permute}
E^*_{s_j \nu}(\bfx;q,t) = 
\left(T_j + \frac{1-t}{1-q^{\nu_{j}-\nu_{j+1}}t^{r_j(\nu)}}\right)
E^*_{\nu}(\bfx;q,t).
\end{equation}
\end{prop}

\begin{thm}[{\cite{Knop1997b,Sahi1996}}]\label{thm:KS_recurrence}
    For any composition $\mu$, we have
 \begin{equation}\label{eq:KS_recurrence}
  E^*_{(\mu_n+1,\mu_1,\dots,\mu_{n-1})}(x_1,\dots,x_n;q,t)=\left(x_1-\frac{1}{t^{n-1}}\right)q^{\mu_n}
  E_\mu^*\left(x_2,\dots,x_n,\frac{x_1}{q};q,t\right).  
    \end{equation}
\end{thm}

\section{The proof of \texorpdfstring{\cref{thm:weak-reordering}}{}}\label{app:permuted_basement}
In this appendix we use analogues of results from \cite{AlexanderssonSawhney2019,AyyerMartinWilliams2025} to 
prove \cref{thm:weak-reordering}. If $\lambda$ is a partition, we will denote by $\lambda'$ \textit{the conjugate} of $\lambda$.
Recall that $e_{\lambda}^*$ is the interpolation elementary symmetric function defined in \cref{def:int_elem}. Moreover, if $\mu=(\mu_1,\dots,\mu_n)$ is a composition, we denote $\mu^{\rev}:=(\mu_n,\dots,\mu_1)$.

\begin{thm}[{\cite[Theorem 18]{AlexanderssonSawhney2019}}]
\label{thm:E_partition_hom}
    For any partition $\lambda$, we have
    $$E^{}_{\lambda^{\rev}}(x_1,\dots,x_n;1,t)=e_{\lambda'}(x_1,\dots,x_n).$$
\end{thm}

We now give  a dehomogenization of this theorem.

\begin{thm}\label{thm:E_partition}
    For any partition $\lambda$, we have
    $$E^*_{\lambda^{\rev}}(x_1,\dots,x_n;1,t)=e^*_{\lambda'}(x_1,\dots,x_n;t).$$
\end{thm}
\begin{proof}
From \cref{prop:Vstar}, we can write an expansion
$$E_{\lambda^{\rev}}(\bfx;q,t)=\sum_{\mu\in S_n(\lambda)}d_{\mu}(q,t)F_\mu(\bfx;q,t),$$
for some coefficients $d_\mu\in \QQ(q,t)$.
Applying the dehomogenization map \cref{lem:dehomog}, we get
\begin{equation}\label{eq:E_F}
  E^*_{\lambda^{\rev}}(\bfx;q,t)=\sum_{\mu\in S_n(\lambda)}d_{\mu}(q,t)F^*_\mu(\bfx;q,t).  
\end{equation}
    But we know from \cref{prop:symmetrization} and \cref{thm:factorization} that
    $$e^*_{\lambda'}(\bfx;t)=\sum_{\mu\in S_n(\lambda)}F^*_\mu(\bfx;1,t).$$
    Combining this with  \cref{thm:E_partition},       we get
    $$E^*_{\lambda^{\rev}}(\bfx;1,t)=\sum_{\mu\in S_n(\lambda)}F^*_\mu(\bfx;1,t).$$
    Comparing this with \cref{eq:E_F}, and the fact that, interpolation ASEP $F_\mu^*$ polynomials are a basis (see \cref{rmk:basis_1}), we deduce that $d_\mu(1,t)=1$ for any $\mu\in S_n(\lambda)$ as desired.
\end{proof}

\begin{thm}\label{thm:recoloring}
Fix a weakly order-preserving function $\phi:\NN\rightarrow\NN$, and two partitions $\lambda$ and $\kappa$ such that $\phi(\lambda)=\kappa$.
Consider a composition  $\eta\in S_n(\kappa)$. Let $\mu \in S_n(\lambda)$ be the lexicographically smallest composition such that $\phi(\mu)=\eta$ (equivalently, $\mu \in S_n(\lambda)$ such that $\phi(\mu)=\eta$ and parts that project to the same integer appear in increasing order in $\mu$).
 We then have
\begin{equation}\label{eq:adjacent}
    \frac{E^*_\mu(\bfx;1,t)}{E^*_{\eta}(\bfx;1,t)}=\frac{e^*_{\lambda'}(\bfx;t)}{e^*_{\kappa'}(\bfx;t)}.
\end{equation}
\end{thm}

\begin{proof}
We prove the result for any $\lambda$ and $\kappa$, and for any composition $\mu \in S_n(\lambda)$ and $\eta=\phi(\mu)\in S_n(\kappa)$ satisfying the conditions of the theorem. We proceed by induction on the size of $\kappa$.  If $\kappa=(0,\dots,0)$, then $\eta=(0,\dots,0)$, and we have that $\mu$ is weakly increasing ($\mu^{\rev}$ is a partition). As a consequence, applying \cref{thm:E_partition}, we get
\begin{equation}
    \frac{E^*_\mu(\bfx;1,t)}{E^*_{\eta}(\bfx;1,t)}=e^*_{\lambda'}(\bfx;t)=\frac{e^*_{\lambda'}(\bfx;t)}{e^*_{\kappa'}(\bfx;t)}.
\end{equation}

We now fix $\kappa$ and $\lambda$, and we write $\kappa=(a_1^{m_1},\dots,a_\ell^{m_\ell})$, for some $m_i>0$, $a_1>0$ and $a_1>a_2>\dots>a_\ell\geq 0$. We want to prove the result for any $\mu\in S_n(\lambda)$ and $\eta\in S_n(\kappa)$ satisfying the condition of the theorem. It is enough to prove this for $\eta=\kappa$, and the other cases will follow by the shape-permuting formula \cref{prop:shape}. In this case, we know that $\mu$ is of the form 
$\mu=(\mu^{(1)},\dots,\mu^{(\ell)})$, where $\mu^{(i)}$ is a weakly increasing sequence of length $m_i$, and if $i<j$ then all parts of $\mu^{(i)}$ are greater than all parts of $\mu^{(j)}$ (the projection $\phi$ sends then all parts of $\mu^{(i)}$ to $a_i$, for $1\leq i\leq \ell$).
We want to show that the ratio $\frac{E^*_\mu(\bfx;1,t)}{E^*_{\kappa}(\bfx;1,t)}$ can be written in the form of \cref{eq:adjacent}. But using the Knop–Sahi recurrence (\cref{thm:KS_recurrence}), it is enough to prove this result for
$$\frac{E^*_{\nu}(\bfx;1,t)}{E^*_{\zeta}(\bfx;1,t)},$$
where $\nu=(\mu^{(2)},\dots,\mu^{(\ell)},\mu^{(1)}-\mathbf{1})$ and $\zeta=(a_2^{m_2},\dots,a_\ell^{m_\ell},(a_1-1)^{m_1})$.

Now there are two cases:
\begin{itemize}
    \item if $a_1-1=a_2$, we check that in this case we can find a projection $\phi'$ sending $\nu$ on $\zeta$, and that these compositions satisfy the conditions of the theorem. 
    More precisely, we choose $\phi'$ such that it sends all the parts of $\mu^{(i)}$ to $a_i$ for $i\geq 2$, and the parts of $\mu^{(1)}-\mathbf{1}$ to $a_1-1=a_2$ (note that the assumption $a_1-1=a_2$ is crucial for such projection to exist, since the parts of $\mu^{(1)}-\mathbf{1}$ might intersect with the parts of $\mu^{(2)}$). 
    Since $|\zeta|<|\kappa|$,  we can apply the induction hypothesis.

    \item if $a_1-1>a_2$, we can write
    $$\frac{E^*_{\nu}(\bfx;1,t)}{E^*_{\zeta}(\bfx;1,t)}=
    \frac{E^*_{\nu}(\bfx;1,t)}{E^*_{(a_2^{m_2},\dots,a_\ell^{m_\ell},a_2^{m_1})}(\bfx;1,t)}\cdotp
    \frac{E_{(a_2^{m_2},\dots,a_\ell^{m_\ell},a_2^{m_1})}(\bfx;1,t)}{E^*_{\zeta}(\bfx;1,t)}$$
    and we apply the induction hypothesis separately on the first term and on the inverse of the second term.\qedhere
\end{itemize}
\end{proof}

In the following, we fix a weakly order-preserving function $\phi:\NN\rightarrow\NN$, and two partitions $\lambda$ and $\kappa$ such that $\phi(\lambda)=\kappa$.

\begin{definition}\label{def:inc}
Fix $\phi$, $\lambda$ and $\kappa$ as above.
Let $\inc(\lambda,\phi)$ be the lexicographically smallest $\rho\in S_n(\lambda)$ such that $\phi(\rho)=\kappa$. 
\end{definition}

\begin{thm}\label{thm:E_inc_E}
    Fix $\kappa$ and $\lambda$ as in \cref{def:inc}. At $q=1$, $E^*_{\inc(\lambda,\phi)}$ is a symmetric function multiple of $E^*_\kappa=F^*_\kappa$.
    \end{thm}
 
\begin{proof}
    We apply \cref{thm:recoloring} with $\eta=\kappa$ and $\mu=\inc(\lambda,\phi)$.
\end{proof}

Recall that for $\eta\in S_n(\kappa)$, 
\begin{equation}\label{eq:G}
G^*_\eta(\bfx;t):=\sum_{\rho\in S_n(\lambda):\, \phi(\rho)=\eta} F^*_\rho(\bfx;1,t),
\end{equation}
and that $G_{\eta}$ is the top homogeneous part of $G^*_\eta$.

\begin{prop}[{\cite[Proposition 4.16]{AyyerMartinWilliams2025}}]\label{prop:hom_G}
    We have 
$G_\kappa(\bfx;t)=E_{\inc(\lambda,\phi)}(\bfx;1,t).$
\end{prop}
We then deduce the following theorem for interpolation polynomials.
\begin{prop}\label{prop:G_E}
    We have 
$G^*_\kappa(\bfx;t)=E^*_{\inc(\lambda,\phi)}(\bfx;1,t).$
\end{prop}
\begin{proof} By definition $G^*_\kappa$ is a linear combination of $F^*_\nu(\bfx;1,t)$ for $\nu\in S_n(\lambda)$. But we know from \cref{prop:Vstar} that for $\nu\in S_n(\lambda)$, $F^*_\nu(\bfx;1,t)$ is a linear combination of $E^*_\rho(\bfx;1,t)$ for $\rho\in S_n(\lambda)$. 
We find the coefficients of this expansion by taking the top homogeneous part and using \cref{prop:hom_G}.
\end{proof}

\begin{proof}[Proof of \cref{thm:weak-reordering}]
We want to prove that for all $\eta\in S_n(\kappa)$, we have at $q=1$
    $$\frac{G^*_\eta(\bfx;t)}{P^*_\lambda(\bfx;1,t)}=\frac{F^*_\eta(\bfx;1,t)}{P^*_\kappa(\bfx;1,t)}.$$
    We know from \cref{prop:G_E} that $G_\kappa^*(\bfx;t)=E^*_{\inc(\lambda,\phi)}(\bfx;1,t)$. Using \cref{thm:E_inc_E}, we get that 
    \begin{equation}\label{eq:G-F1}
      G_\kappa^*(\bfx;t)
      =h\cdotp E_\kappa^*(\bfx;1,t)
      =h\cdotp F_\kappa^*(\bfx;1,t),  
    \end{equation}
    for some symmetric function $h$.
    Now by \cref{prop:T_i-F_star}, for any
    $\eta \in S_n(\kappa)$, we have $F^*_\eta=T_\sigma \cdotp F_\kappa^*,$
    where $\sigma$ is the shortest permutation sending $\kappa$ to $\eta$. One then checks that we also have $G^*_\eta=T_\sigma \cdotp G_\kappa^*$ (see the proof of \cite[Lemma 4.12]{AyyerMartinWilliams2025} for more details).
    Now applying $T_\sigma$ to \cref{eq:G-F1} we get that 
    $ G_\eta^*(\bfx;t)=h\cdotp F_\eta^*(\bfx;1,t).$
    Summing this equation over all $\eta$, we get
    $$P_\lambda^*(\bfx;1,t)=h\cdotp P_\kappa^*(\bfx;1,t),$$
    where we have used \cref{eq:G} and \cref{prop:symmetrization}. We then get the desired claim
    \begin{equation*}
    \frac{P_\lambda^*(\bfx;1,t)}{P_\kappa^*(\bfx;1,t)}=h=\frac{G_\eta^*(\bfx;t)}{F_\eta^*(\bfx;1,t)}.\qedhere    
    \end{equation*}
\end{proof}

\section{\texorpdfstring{$t$}{t}-interpolation Schur polynomials}\label{app:t_int_Schur}
Throughout this section we fix a positive integer $n$,
and write 
$g(\bfx)=g(x_1,\dots,x_n)$.
Recall that the $t$-interpolation Schur polynomial $s_\lambda^*$ is the symmetric polynomial defined by
    $$s^*_\lambda(\bfx;t)=P_\lambda(\bfx;t,t).$$

We let $\SSYT(\lambda;n)$ denote the set of \emph{semistandard Young tableaux} of shape $\lambda$ filled with values in $\llbracket n\rrbracket$. For a cell $\Box=(i,j)$ in the Young diagram of $\lambda$ in  row $i$ and column $j$, we let $c(\Box)=j-i$ denote the \emph{content} of $\Box$.

With the specialization $q=t$, Okounkov's formula for interpolation Macdonald polynomials \cite[Equation (1.4)]{Okounkov1998} can be written as follows.

\begin{thm}[\cite{Okounkov1998}]\label{thm:Okounkov}
For any partition $\lambda=(\lambda_1,\dots,\lambda_n)$, we have
$$s^*_\lambda(\bfx;t)=\sum_{T\in \SSYT(\lambda)}\prod_{\Box\in \lambda}(x_{T(\Box)}-t^{c(\Box)+T(\Box)-n}).$$
\end{thm}

For $n\geq k\geq 0$ and $\ell\in \ZZ$, we define
$$e_{k,\ell}^*(\bfx;t):=\sum_{\begin{subarray}{c}S\subseteq \llbracket n \rrbracket\\  \#S=k\end{subarray}}\prod_{i\in S}\left(x_i-\frac{t^{\#S^c\cap \llbracket i-1\rrbracket}}{t^{\ell}}\right),$$
so that $e^*_k(\bfx;t)=e^*_{k,n-1}(\bfx;t)$
and also
\begin{equation}\label{eq:identity}
e^*_{k,n-2}(tx_1,\dots,tx_n;t) = t^k e^*_{k,n-1}(x_1,\dots,x_n;t).    
\end{equation}

We then have the following $t$-deformation of the dual Jacob–Trudi identity. It follows from \cref{thm:Okounkov}, together with the results of \cite{Macdonald1992} on the 6th variation of Schur functions (more precisely, with the notation \cite[Section6]{Macdonald1992}, we set $a_i=-t^{i-n}$).

\begin{prop}[\cite{Macdonald1992}]
    Let $\lambda=(\lambda_1,\dots,\lambda_{n})$ be a partition. We have
    \begin{equation}\label{eq:int_dual_JT}
      s^*_\lambda(\bfx;t)=\det\left(e^*_{\lambda'_i-i+j, n-j}(\bfx;t)\right)_{1\leq i,j\leq n}.  
    \end{equation}
\end{prop}

\begin{remark}
    Note that the top homogeneous part of \cref{eq:int_dual_JT} gives the dual Jacobi–Trudi identity (see \textit{e.g} \cite[Corollary 7.16.2]{Stanley:EC2}). As in the homogeneous case, this result can also be obtained using the Lindström–Gessel–Viennot lemma.
\end{remark}

If $\lambda$ is a 2-column partition, then \cref{eq:identity} and \cref{eq:int_dual_JT} give 
\cref{lem:2column}.

\end{appendices}

\bibliographystyle{amsalpha}
\bibliography{biblio.bib}
   
\end{document}